\documentclass[11pt]{amsart}
\usepackage{amssymb}
\usepackage{amsmath}

\newcommand{\ck}{{\bar{\bk}}}

\newcommand{\f}{\textbf}

\newtheorem{numbered}{}[section]
\newtheorem{Theorem}[numbered]{Theorem}

\newtheorem{nothing*}[numbered]{}
\newtheorem{remark}[numbered]{Remark}

\newtheorem{Definition}[numbered]{Definition}

\newtheorem{Lemma}[numbered]{Lemma}
\newtheorem{Proposition}[numbered]{Proposition}

\newtheorem{example}[numbered]{Example}

\def\bn{{\mathbb N}}

\def\br{{\mathbb R}}

\def\l{\lambda}

\def\i{\varepsilon}

\def\f{\varphi}

\def\N{\mathbb{N}}

\def\a{\alpha}

\def\b{\beta}
\def\m{\mu}

\def\d{\delta}
\def\g{\gamma}

\def\ca{\mathcal{ A}}

\def\ck{\mathcal{K}}

\def\cq{\mathcal{ Q}}

\def\cs{\mathcal{ S}}
\def\ct{\mathcal{ T}}

\def\ga{{\frak A}}

\def\id{{\bf 1}\!\!{\rm I}}

\begin{document}

\begin{center}
{\Large {\bf Generalized Dobrushin Ergodicity Coefficient and Ergodicities of Non-homogeneous Markov Chains}}\\[1cm]

{\sc Farrukh Mukhamedov} \\[2mm]

 Department of Mathematical Sciences, College of Science, \\
 United Arab Emirates University 15551, Al-Ain,\\ United
Arab Emirates.\\

e-mail: {\tt far75m@gmail.com; farrukh.m@uaeu.ac.ae}\\[1cm]

{\sc  Ahmed Al-Rawashdeh} \\[2mm]

Department of Mathematical Sciences, College of Science, \\
United Arab Emirates University 15551, Al-Ain, \\ United
Arab Emirates.\\
 e-mail: {{\tt aalrawashdeh@uaeu.ac.ae}}\\[1cm]
\end{center}

\small
\begin{center}
{\bf Abstract}\\
\end{center}

In our earlier paper, a generalized Dobrushin ergodicity
coefficient of Markov operators (acting on abstract state spaces)
with respect to a projection $P$, has been introduced and studied.
It turned out that the introduced coefficient was more effective
than the usual ergodicity coefficient. In the present work, by
means of a left consistent Markov projections and the generalized
Dobrushin's ergodicity coefficient, we investigate uniform and
weak $P$-ergodicities of non-homogeneous discrete Markov chains
(NDMC) on abstract state spaces. It is easy to show that uniform
$P$-ergodicity implies a weak one, but in general the reverse is
not true. Therefore, some conditions are provided together with
weak $P$-ergodicity of NDMC which imply its uniform
$P$-ergodicity. Furthermore, necessary and sufficient conditions
are found by means of the Doeblin's condition for the weak
$P$-ergodicity of NDMC. The weak $P$-ergodicity is also
investigated in terms
 of perturbations. Several perturbative results are obtained which allow us to produce
nontrivial examples of uniform and weak $P$-ergodic NDMC.
Moreover, some category results are also obtained.  We stress that
all obtained results have potential applications in the classical
and non-commutative probabilities. \\[2mm]
\textit{MSC}: 47A35; 60J10, 28D05\\
\textit{Key words}:  uniform $P$-ergodic; weak $P$-ergodic; Markov
operator; projection; ergodicity coefficient; non-homogeneous
discrete Markov chain;

\normalsize

\section{Introduction}

The present work is a continuation of the paper \cite{MA} where we
have introduced a generalized Dobrushin ergodicity coefficient
$\d_P(T)$ of Markov operators (acting on abstract state spaces)
with respect to a projection $P$, and studied its properties. It
turns out that the introduced coefficient was more effective than
the usual Dobrushin's ergodicity coefficient \cite{D}. Indeed,
uniform stability of the trajectories of Markov operator to some
projection have been investigated by means of the generalized
Dobrushin's ergodicity coefficient, while usual ergodicty
coefficient is not applicable in that situation.  We stress that
in the literature, much attention is paid to homogeneous Markov
processes (see \cite{B,BK,BK1,EW1,J1,SZ}). However, a limited
number of papers (see \cite{EM20181,M2013,M0,M01,VW}) were devoted
to investigations of ergodic properties of nonhomogeneous Markov
processes in the abstract scheme. In those papers the limiting
operators are considered as one-dimensional projections.
Therefore, one of the aims of the present paper is to investigate
stabilities of nonhomogeneous discrete Markov chains (NDMC) to
some projection in the abstract framework. We notice that this
abstract scheme contains both classical and quantum settings as
particular cases \cite{Alf,E} which impies that the obtained
results will be new in both cases. In this abstract setting,
certain limiting behaviors of Markov operators were investigated
in \cite{B,EW2,GQ,SZ}.

It is stressed that due to the nonhomogenety of Markov processes,
the investigations of limiting behavior of such processes become
very complicated. The Dobrushin's ergodicity coefficient was
effectively used in the investigation of ergodicity of
nonhomogeneous Markov chains when their limit is a one-dimensional
projection (see \cite{D,HB,HH,IS,J1,Paz,P,Rod,Se,T,ZI}). If the
limit of the chain is a more general projection then the usual
coefficient is not effective. Therefore, using the generalized
coefficient, we are going to investigate uniform and weak
ergodicities of NDMC in the abstract state spaces. In \cite{H,HR}
a very simple case of NDMC (on finite dimensional spaces) has been
investigated when the limit is a projection. We point out that all
obtained results will be new for classical and quantum Markov
chains.

As it is mentioned that our purpose is to investigate the
stability (in uniform and weak topologies) of NDMC acting on
abstract state spaces. Here, by an abstract state space it is
meant an ordered Banach space, where the norm has an additivity
property on the cone of positive elements. Examples of these
spaces include all classical $L^1$-spaces and the space of density
operators acting on some Hilbert spaces \cite{Alf,Jam}. Moreover,
any Banach space can be embedded into some abstract spaces (see
Appendix, Example \ref{E13}). In the present paper, we are going
to study the asymptotic stabilities  of NDMC in terms of
generalized Dobrushin's ergodicity coefficient.  We notice that
the Dobrushin coefficient (which extends $\d(P)$ to abstract state
spaces) has been introduced and studied in \cite{GQ,M0,M01}, for
Markov operators acting on abstract state spaces.

Let us briefly describe the organization of the paper. In Section
2, we provide preliminary definitions and results on properties of
the generalized Dobrushin's ergodicty coefficient. Moreover, the
uniform and weak $P$-ergodicities of NDMC are defined. In Section
3, we introduce sequences of left consistent Markov projections
which will be used in the forthcoming sections. We note that if
the chain is homogeneous then the uniform and weak
$P$-ergodicities coincide \cite{MA}, but in the non-homogeneous
setting the situation is much more complicated \cite{Paz,Se,T}.
Therefore, in Section 4, certain relations between these notions
are going to be investigated. The results of this section are
known facts related to NDMC \cite{Paz,M0} when the limiting
projection is one-dimensional. In Section 5, necessary and
sufficient conditions are established by means of the Doeblin's
condition for the weak $P$-ergodicity of NDMC. These extend the
results of \cite{DP,M0,M01} to an abstract scheme. Note that in
\cite{DP} similar conditions were found for classical
nonhomogeneous Markov processes to satisfy weak ergodicity. In
Section 6, we study the weak $P$-ergodicity in terms of
perturbations. Several perturbation results are obtained which
allow us to produce nontrivial examples of uniform and weak
$P$-ergodic NDMC. Finally, in Section 7 some category results are
also obtained. Namely, in the set of all NDMC (consistent with a
left decreasing sequence of projections $\{P_n\}$)
$\frak{S}_{\{P_n\}}(X)$, we introduce a mertric, according to the
set $\frak{S}_{\{P_n\}}^{w}(X)$ of all weak ergodic w.r.t.
$\{P_n\}$ NDMC is a $G_\d$-dense subset of
$\frak{S}_{\{P_n\}}(X)$. Some similar kinds of results have been
proved in \cite{P} when $X=\ell^1$. We stress that all obtained
results have potential applications in the classical and
non-commutative probabilities.

\section{Preliminaries}

In this section, we recall some necessary definitions and results
about abstract state spaces.

 Let $X$ be an ordered vector space
with a cone $X_+=\{x\in X: \ x\geq 0\}$. A subset $\ck$ is called
a {\it base} for $X$, if  $\ck=\{x\in X_+:\ f(x)=1\}$ for some
strictly positive (i.e. $f(x)>0$ for $x>0$) linear functional $f$
on $X$. An ordered vector space $X$ with generating cone $X_+$
(i.e. $X=X_+-X_+$) and a fixed base $\ck$, defined by a functional
$f$, is called {\it an ordered vector space with a base}
\cite{Alf}.  Let $U$ be the convex hull of the set
$\ck\cup(-\ck)$, and let
$$
\|x\|_{\ck}=\inf\{\l\in\br_+:\ x\in\l U\}.
$$
Then one can see that $\|\cdot\|_{\ck}$ is a seminorm on $X$.
Moreover, one has $\ck=\{x\in X_+: \ \|x\|_{\ck}=1\}$,
$f(x)=\|x\|_{\ck}$ for $x\in X_+$. Assume that the seminorm
becomes a norm and $X$ is a complete space w.r.t. this norm and
$X_+$ is closed subset, then $(X,X_+,\ck,f)$ is called
\textit{abstract state space}. In this case, $\ck$ is a closed
face of the unit ball of $X$, and $U$ contains the open unit ball
of $X$. If the set $U$ is \textit{radially compact} \cite{Alf},
i.e. $\ell\cap U$ is a closed and bounded segment for every line
$\ell$ through the origin of $X$, then  $\|\cdot\|_{\ck}$ is a
norm. The radial compactness is equivalent to the coincidence of
$U$ with the closed unit ball of $X$.  In this case, $X$  is
called a \textit{strong abstract state space}.  In the sequel, for
the sake of simplicity, instead of $\|\cdot\|_{\ck}$, the standard
notation $\|\cdot\|$ is used. For a better understanding of the
difference between a strong abstract state space and a more
general class of base norm spaces, the reader is referred to
\cite{Yo}.

A positive cone $X_+$  of an ordered Banach space $X$  is said to
be $\l$-generating if, given $x\in X$, we can find $y,z\in X_+$
such that $x = y- z$ and $\|y\| + \|z\|\leq\l \|x\|$. The norm on
$X$ is  called \textit{regular} (respectively, \textit{strongly
regular}) if, given $x$ in the open (respectively, closed) unit
ball of $X$, $y$ can be found in the closed unit ball with $y\geq
x$ and $y\geq -x$. The norm is said to be additive on $X_+$ if
$\|x + y\| = \|x\| + \|y\|$ for all $ x, y\in X_+$. If $X_+$ is
1-generating, then  $X$ can be shown to be strongly regular.
Similarly, if $X_+$ is $\l$-generating for all $\l > 1$, then $X$
is regular \cite{Yo}. The following results are well-known.

\begin{Theorem}\cite[p.90]{WN} Let  $X$ be an ordered Banach space with closed positive cone $X_+$.
Then te following statements are equivalent:
\begin{itemize}
\item[(i)] $X$ is an abstract state space; \item[(ii)] $X$ is
regular, and the norm is additive on $X_+$; \item[(iii)] $X_+$ is
$\l$-generating for all $\l > 1$, and the norm is additive on
$X_+$.
\end{itemize}
\end{Theorem}

\begin{Theorem}\cite{Yo}\label{Yo} Let  $X$ be an ordered Banach space with closed positive cone $X_+$.
Then the following statements are equivalent:
\begin{itemize}
\item[(i)] $X$ is a strong abstract state space; \item[(ii)] $X$
is strongly regular, and the norm is additive on $X_+$;
\item[(iii)] $X_+$ is 1-generating and the norm is additive on
$X_+$.
\end{itemize}
\end{Theorem}

In this paper, we consider a general abstract state space for
which the convex hull of the base $\ck$ and $-\ck$ is not assumed
to be radially compact (in our previous papers
\cite{EM2017,EM2018,M0,M01} this condition was essential). This
consideration has an important advantage: whenever $X$ is an
ordered Banach space with a generating cone $X_+$ whose norm is
additive on $X_+$, then $X$ admits an equivalent norm that
coincides with the original norm on $X_+$ and renders $X$ that
base norm space. Hence, to apply the results of the paper one
would only have to check if the norm is additive on $X_+$.

Let $(X,X_+,\ck,f)$ be n abstract state space. A linear operator $T:X\to X$ is
called \textit{positive}, if $Tx\geq 0$ whenever $x\geq 0$. A positive linear
operator $T:X\to X$ is said to be {\it Markov}, if $T(\ck)\subset\ck$.
It is clear that $\|T\|=1$, and its adjoint mapping $T^*: X^*\to
X^*$ acts in ordered Banach space $X^*$ with unit $f$, and moreover,
one has $T^*f=f$.

Recall that a family of Markov operators $\{T^{m,n}: X\to X\}$
($m\leq n$, $m,n\in\N$) is called a \textit{non-homogeneous
discrete Markov
 chain (NDMC)} if
$$
T^{m,n}=T^{k,n}T^{m,k-1}
$$
for every $m\leq k\leq n$. Due to this property, to any NDMC
$\{T^{m,n}\}$ one can associate a sequence $\{T_n\}_{n=1}^\infty$, (where
$T_n=T^{n,n+1}$) of Markov operators. Conversely, any given a
sequence of Markov operators $\{T_n\}_{n=1}^\infty$ on $X$ and for
$k<n$, by putting
\[T^{k,n}:=T_nT_{n-1}\ldots T_{k+1}.\]
we also can define a NDMC $\{T^{k,n}\}$. This chain is generated
by $\{T_n\}$, such a sequence $\{T_n\}$ is called
\textit{generating sequence} of the NDMC. Therefore, NDMC
$\{T^{k,n}\}$ can be identified with its generating sequence. In
the last section 7, we will use this identification.

Recall that if for a given NDMC $\{T^{k,m}\}$ one has
$T^{k,m}=(T^{0,1})^{m-k}$, then such a chain becomes {\it
homogeneous}. In what follows, by $\{T^n\}$ we denote a homogeneous
Markov chain, where $T:=T^{0,1}$. Equivalently, any NDMC is
homogeneous, if its generating sequence is stationary, i.e.
$T_n=T_1$ for all $n\in\bn$.\\

Let $(X,X_+,\ck,f)$ be an abstract state space and let $\{T^n\}$
be a homogeneous Markov chain on $X$. Consider a projection
operator $P: X\to X$ (i.e. $P^2=P$). According to \cite{MA}
$\{T^n\}$ is called \textit{uniformly $P$-ergodic} if
$$
\lim_{n\to\infty}\|T^{n}-P\|=0.
$$
From this definition we immediately find that $P$ must be a Markov
projection.

Analogously, we say that a NDMC  $\{T^{m,n}\}$ is called {\it
uniformly $P$-ergodic} if for every $m\geq 0$ one has
$$
\lim_{n\to\infty}\|T^{m,n}-P\|=0.
$$

We note that if $P=T_y$, for some $y\in X_+$, where
$T_y(x)=f(x)y$, then the uniform $P$-ergodicity coincides with
uniform ergodicity or uniform asymptotical stability considered in
\cite{M0,M01}.
%
%

In \cite{MA}, we have introduced a generalized notion of the
Dobrushin's ergodicity coefficient as follows:

Let $(X,X_+,\ck,f)$ be an abstract state space and let $T:X\to X$
be a linear bounded operator and $P$ be a non-trivial projection
operator on $X$. Then we define
\begin{equation} \label{Dbp} \d_P(T)=\sup_{x\in N_P,\ x\neq
0}\frac{\|Tx\|}{\|x\|},
\end{equation}
where
\begin{equation}\label{Np} N_P=\{x\in X: \
Px=0\}.
\end{equation}

If $P=I$, we put $ \d_P(T)=1$. The quantity $\d_P(T)$ is called
the \textit{generalized Dobrushin ergodicity coefficient of $T$
with respect to $P$}.

We notice that if $X=\br^n$, then there are some formulas to
calculate this coefficient (see \cite{H,HR}).

In the following remarks, let us have a brief comparison between
the coefficients $\d_P(T)$ and $\d(T)$. It is noticed that $\d(T)$
has been introduced and investigated in \cite{M0,M01}.

\begin{remark}
Let $y_0\in \ck$ and consider the projection $Px=f(x)y_0$. Then
one can see that $N_P$ coincides with
\[N= \{ x\in X;\ f(x)=0\},\]
and in this case $\d_P(T)=\d(T)$. Hence, $\d_P(T)$ indeed is a
generalization of $\d(T)$.
\end{remark}

\begin{remark}
Let $P$ be a Markov projection on $X$. Then, for any Markov
operator $T:X\to X$
\[\d_P(T) \leq \d(T).\]
\end{remark}

Using this coefficient, we define weak $P$-ergodicity of NDMC.
Namely, a NDMC  $\{T^{m,n}\}$ is called {\it weakly $P$-ergodic}
if for every $m\geq 0$ one has
$$
\lim_{n\to\infty}\d_P(T^{m,n})=0.
$$

We point out that the relations between uniform and week
$P$-ergodicities will be discussed in Section 4.

Let us denote by $\Sigma(X)$ the set of all Markov operators
defined on $X$, and by $\Sigma_P(X)$ we denote the set of all
Markov operators $T$ on $X$ with $PT=TP$.\\

We recall certain properties of $\d_P(T)$, which are given in the
following theorem:

\begin{Theorem}\cite{MA}\label{MA1-Thr. 3.7} Let $(X,X_+,\ck,f)$ be an abstract state space, $P$ be a projection on $X$ and let $T,S\in \Sigma(X)$. Then:
\begin{enumerate}
\item[(i)] $0\leq \d_P(T)\leq 1$;
\item[(ii)]  $|\d_P(T)-\d_P(S)|\leq\d_P(T-S)\leq \|T-S\|$;
\item[(iii)] if $P\in \Sigma(X)$,  one has
\begin{equation}\label{dpuv} \d_P(T)\leq \frac{\l}{2}\sup \{\|Tu-Tv\|;\ u,v\in \ck \ \text{with}\ u-v\in N_P \}.
\end{equation}
\item[(iv)]  if $H: X\to X$ is a bounded linear
operator such that $HP=PH$, then $$\d_P(TH)\leq\d_P(T)\|H\|;$$
\item[(v)]   if $H: X\to X$ is a bounded linear
operator such that $PH=0$, then $$\|TH\|\leq\d_P(T)\|H\|;$$
\item[(vi)]  if $S\in \Sigma_P(X)$, then $$\d_P(TS)\leq\d_P(T)\d_P(S).$$
\end{enumerate}
\end{Theorem}

\noindent We stress that the condition $PS=SP$ in (vi)
can be weakened as follows:
\begin{Proposition}\label{dp12-less-1-2}
  Let $(X,X_+,\ck,f)$ be an abstract state space, $P$ be a projection on $X$ and let $T_1$ and $T_2$ be operators on $X$. If $T_2(N_P)\subseteq N_P$, then
  \[\d_P(T_1T_2)\leq\d_P(T_1)\d_P(T_2). \]
\end{Proposition}
\begin{proof}
For all $x\in N_P$ we have $T_2x\in N_P$, so
 \begin{eqnarray*}
\|T_1(T_2x)\| &\leq & \d_P(T_1)\|T_2x\|\\
& \leq & \d_P(T) \d_P(T_2)\|x\|,
\end{eqnarray*}
which implies
\[\frac{\|T_1T_2x\|}{\|x\|}\leq \d_P(T_1)\d_P(T_2),\ \forall \ x\in N_P,\]
then
\[\d_P(T_1T_2)\leq \d_P(T_1)\d_P(T_2),\]
and hence the result follows.
\end{proof}

\begin{Lemma}\label{TNp-in-Npmeans}
  Let $(X,X_+,\ck,f)$ be an abstract state space, $P$ be a projection on $X$ and let $T$ be an operator on $X$. Then $T(N_P)\subseteq N_P$ if and only if
  $PT=PTP$.
\end{Lemma}
\begin{proof}
Assume that $PT=PTP$. If $x\in N_P$, then $P(Tx)=PTP(x)=0$, so we get $Tx\in N_P$. Conversely, suppose that $T(N_P)\subseteq N_P$ and $x\in N_P$ and as $N_P=(I-P)X$,
 so $x=(y-Py)$, for some $y\in X$. Therefore,
 \[ 0= PTx= PT(y-Py)=PTy-PTPy,\]
 which implies that $PT=PTP$.
\end{proof}
\begin{remark}\label{PT-PTP}
  It is easy to check that if $PT=P$ or $PT=TP$, then $PT=PTP$.
\end{remark}

In what follows, we need the following auxiliary fact.

\begin{Lemma}\cite{MA}\label{MA1-Lemma 3.6} Let $(X,X_+,\ck,f)$ be an abstract state space and let $P$ be a Markov projection. Then for every $x\in N_P$ there
exist $u,v\in \ck$ with $u-v\in N_P$ such that
$$
x=\a(x)(u-v),
$$
where $\a(x)\in \br_+$ and $\a(x)\leq \frac{\l}{2}\|x\|$.
\end{Lemma}
%
%
%
%

\section{Left Consistent Projections}

In this section, we are going to study a relation between
projections of $X$.

\begin{Definition}
 Let $P$ and $Q$ be projections on $X$. We say that $P$ is \textit{left consistent} by $Q$, and denoted by $P\leq^\ell Q $, if
 $PQ=P$.
\end{Definition}

\begin{Proposition}
The relation $\leq^\ell$ has the following properties:
\begin{itemize}
\item[(i)] $\leq^\ell$ is reflexive; \item[(ii)] $\leq^\ell$ is
transitive.
\end{itemize}
\end{Proposition}

\begin{proof}
(i) is obvious. To prove (ii), assume that $P_1\leq^\ell P_2$ and
$P_2\leq^\ell P_3$, which implies that $P_1P_2=P_1$ and
$P_2P_3=P_2$. Then
\[P_1P_3=P_1P_2P_3=P_1P_2=P_1,\]
and hence $P_1\leq^\ell P_3$.
\end{proof}

In what follows, one needs the following property of $\d_P(T)$
with the relation $\leq^\ell$.

\begin{Proposition}\label{PQl}
Let $T:X\to X$ be a linear bounded operator. If $P$ and $Q$ are
two projections on $X$ such that $P\leq^\ell Q$, then $\d_Q(T)\leq
\d_P(T)$.
\end{Proposition}
\begin{proof}
Assume that $P\leq^\ell Q$. Then for every $x\in N_Q$ we get
$Px=PQx=0$, therefore $N_Q\subseteq N_P$ which yields the desired
inequality.
\end{proof}

A sequence $\{P_n\}_{n=1}^\infty$ of projections of $X$ is called
\textit{left decreasing}, if $P_{n+1}\leq^\ell P_n$, for all $n\in
\bn$, i.e. $P_{n+1}$ is left consistent by $P_n$.

 \begin{example}\label{LCE1}
 Let $(X,X_+,\ck,f)$ be an abstract state space and let $\{z_n\}$ be a sequence in $\ck$ such that $z_n\rightarrow z$. Construct the one dimensional projection $P_n:=T_{z_n}$, i.e. $P_nx=f(x)z_n$. Then
 $\{P_n\}$ is a left decreasing sequence of projections, indeed
 \begin{eqnarray*}
   P_{n+1}P_nx &=& P_{n+1}(f(x)z_n) \\
    &=& f(x)P_{n+1}(z_n) \\
    &=& f(x)f(z_n)z_{n+1} \\
    &=& f(x)z_{n_+1}\\
    &=& P_{n+1}(x).
 \end{eqnarray*}
 Moreover $P_n\rightarrow P=T_z$
 \end{example}
\noindent  In particular, we consider the following example:

\begin{example}
Consider the space $\ell_1$, and recall that $\ck=\{x\in \ell_1 ;\ \sum_{n=1}^\infty x _n=1, x_n\geq 0\}$. For all $n\in \bn$, define
\[z_n=\left (\frac{1}{2},\frac{1}{2^2}, \frac{1}{2^3}, \ldots \frac{1}{2^n}, \frac{1}{2^n}, 0, 0, \ldots 0\right ).\]
Then it is clear that $z_n\in \ck$ and $z_n\rightarrow z=\left (\frac{1}{2},\frac{1}{2^2}, \frac{1}{2^3}, \ldots  \right ) $, as $\|z_n-z\|=\sum_{k=n+1}^\infty \frac{1}{2^k}\rightarrow 0$. Due to the previous example, the projections $P_n=T_{z_n}$ form a left decreasing sequence.

\end{example}

Next result gives some important properties of left consistent
sequences of projections.

\begin{Lemma}\label{PP}
 Let $\{P_n\}_{n=1}^\infty$ be a sequence of projections of $X$.
 Then the following statements hold:
 \begin{itemize}
\item[(i)]  if $\{P_n\}_{n=1}^\infty$ is left decreasing, then
$P_m\leq^\ell P_k$, for all $m\geq k$; \item[(ii)] if
$\{P_n\}_{n=1}^\infty$ is left decreasing and $P_n\rightarrow P$
in norm, as $n\rightarrow \infty$, then $P$ is a projection and
$P\leq^\ell P_k$, for all $k\in \bn$.
 \end{itemize}
\end{Lemma}

\begin{proof}
(i)  Assume that $\{P_n\}_{n=1}^\infty$ is a left decreasing sequence. Then for all $m\geq k$, we have
  \begin{eqnarray*}
    P_mP_k &=& P_mP_{m-1}P_k \\
 &=&  P_mP_{m-1}P_{m-2}P_k   \\
 & \vdots & \\
 &=& P_mP_{m-1}P_{m-2}\ldots P_{k+1}P_k \\
 &=& P_mP_{m-1}P_{m-2}\ldots P_{k+1} \\
 & \vdots & \\
 &=&  P_mP_{m-1} \\
     &=& P_m.
  \end{eqnarray*}
  (ii) is obvious.
\end{proof}

%
%
%
%

\section{Uniform $P$-ergodicity  and weak $P$-ergodicity of NDMC}

In this section, we discuss some relations between weak and uniform
$P$-ergodicities of NDMC. The following result show that weak
$P$-ergodicity is indeed weaker that the uniform $P$-ergodicity.


\begin{Proposition} Let $X$ be an abstract state space. Then every uniformly $P$-ergodic NDMC is weakly $P$-ergodic.
\end{Proposition}

\begin{proof}
By Theorem \ref{MA1-Thr. 3.7}(iii), we have
\begin{eqnarray*}
\d_P(T^{k,n}) &\leq& \frac{\l}{2}\sup \|T^{k,n}u-T^{k,n}v\| \ \ \ (\ u,v\in \ck,\ \text{and}\ Pu=Pv)\\
&=& \frac{\l}{2}\sup \|T^{k,n}u-Pu+Pv - T^{k,n}v\|\\
&\leq & \frac{\l}{2}( \sup \|T^{k,n}u-Pu\|+ \sup \|T^{k,n}v-Pv \|)\\
&\leq & \l \|T^{k,n}-P\|\rightarrow 0,\ \text{as}\ n\rightarrow
\infty
\end{eqnarray*}
and hence $\d_P(T^{k,n})\rightarrow 0$ as $n\rightarrow \infty$, which completes the proof.
\end{proof}

Before discussing the reverse direction, let us consider some
examples of uniform $P$-ergodic NDMC.

\begin{example}
 Let $(X,X_+,\ck,f)$ be an abstract state space. Let $\{\tilde{T}_n\}_{n=1}^\infty$ be a sequence of Markov operators such that
 $\bigcap\limits_{n=1}^\infty Fix(\tilde{T}_n)\neq \emptyset$
 (here $Fix(T)=\{x\in\ck: \ Tx=x\}$).
 For a given $a\in (0,1)$ and $z_0\in \ck$ with $z_0\in \bigcap\limits_{n=1}^\infty Fix(\tilde{T}_n)$, define the operators
 \[ T_n=aT_{z_0}+ (1-a)\tilde{T}_n. \]
 One can see that
 \begin{eqnarray*}
   T_nT_m &=& (aT_{z_0}+ (1-a)\tilde{T}_n)( aT_{z_0}+ (1-a)\tilde{T}_m) \\
    &=& a^2T_{z_0} + a(1-a)T_{z_0}\tilde{T}_m + a(1-a)\tilde{T}_nT_{z_0} + (1-a)^2\tilde{T}_n\tilde{T}_m \\
    &=& a^2T_{z_0} + a(1-a)T_{z_0} + a(1-a)T_{z_0}+ (1-a)^2\tilde{T}_n\tilde{T}_m \\
    &=& (1-(1-a)^2)T_{z_0}+ (1-a)^2\tilde{T}_n\tilde{T}_m.
 \end{eqnarray*}
 Hence, for all $k<n$
 \begin{equation*}
   T^{k,n} = T_nT_{n-1}\ldots T_{k+1}= (1-(1-a)^{n-k})T_{z_0}+ (1-a)^{n-k}\tilde{T}_n\ldots \tilde{T}_{k+1}.
 \end{equation*}
    Therefore,
   \begin{equation*}
  \| T^{k,n}-T_{z_0}\| \leq  (1-a)^{n-k} \|T_{z_0}-\tilde{T}_n\ldots \tilde{T}_{k+1}\|\leq 2(1-a)^{n-k}  \to 0,
 \end{equation*}
as $n\to \infty$, which proves that $\{T^{k,n}\}$ is uniformly $P$-ergodic, where $P=T_{z_0}$.
\end{example}

It is interesting to find some conditions which together with weak
$P$-ergodicity of NDMC imply its uniform $P$-ergodicity.

\begin{Theorem}\label{4.1} Let $(X,X_+,\ck,f)$ be an abstract state space and $\{T_n\}$ be
a generating sequence of NDMC. Let $\{P_n\}$ be a sequence of
projections of $X$ such that
\begin{enumerate}
\item[(i)] $T_nP_n=P_nT_n=P_n, \forall n\in \mathbb{N}$,
\item[(ii)] $\{P_n\}$ is left decreasing sequence of projections,
 \item[(iii)] $\sum_{n=k}^\infty\|P_{n+1}-P_n\|\rightarrow 0$, as $k\rightarrow
\infty$. \end{enumerate}
 If $\{T^{k,n}\}$ is weakly
$P$-ergodic, then it is uniformly $P$-ergodic.
\end{Theorem}

\begin{proof}
From the hypotheses (iii) and using the standard argument, one
finds that $P_n\rightarrow P$ (in norm), where $P$ is a
projection. Due to (ii) Lemma \ref{PP} we have $P\leq^\ell P_k$,
for all $k\in \bn$. As
\[ \|T^{m,n}-P \|\leq \|T^{m,n}-P_{n} \| +  \|P_{n}-P \|\]
for $m<n$, it is enough to prove that $\|T^{m,n}-P_{n}
\|\rightarrow 0$, as $n\rightarrow \infty$.

According to (i), one has
\begin{eqnarray}\label{00Eqn4}
  T^{k,n}P_{k+1} &=& T^{k+1,n}T_{k+1}P_{k+1} \nonumber \\
   &=& T^{k+1,n}P_{k+1} \nonumber \\
   &=& T^{k+1,n}(P_{k+1}-P_{k+2})+ T^{k+1,n}P_{k+2}\nonumber \\
   & \vdots & \nonumber\\
   &=& \sum_{l=k+1}^{n-1}T^{l,n}(P_l-P_{l+1})+ P_{n}.
\end{eqnarray}
For any $\i>0$, and by (iii), there is $k_0\in\bn$ (without lost
of generality we may assume that $k_0>m$) such that
\begin{equation}\label{0Eqn4}
\sum_{l=k_0}^{\infty}\| P_l-P_{l+1}\|<\i .
\end{equation}
Therefore, from \eqref{00Eqn4} and \eqref{0Eqn4}, we have
\begin{equation}\label{Eqn4}
\| T^{k_0,n}P_{k_0+1}- P_{n}\| \leq
\sum_{l=k_0}^{n-1}\|T^{l+1,n}(P_l-P_{l+1})\| \leq
\sum_{l=k_0}^{\infty}\| P_l-P_{l+1}\| <\i.
\end{equation}
For all $n\in \bn$, and by (i), one finds
\begin{equation}\label{Eqn5}
 PT_n=PP_nT_n=PP_n=P,
 \end{equation}
which implies that $PT^{m,k_0-1}=P$, for all $m<k_0$, and hence
$P(T^{m,k_0-1}-P_{k_0+1})=0$. So, (v) of Theorem (\ref{MA1-Thr.
3.7}) implies
\begin{equation}\label{Eqn6}
 \| T^{k_0,n}(T^{m,k_0-1}-P_{k_0+1})\|\leq \d_P(T^{k_0,n} )\| T^{m,k_0-1}-P_{k_0+1} \|\leq 2\d_P(T^{k_0,n}).
\end{equation}
Due to the weak $P$-ergodicity of $\{T^{k,n}\}$, there is
$N_0\in\bn$ such that $\d_P(T^{k_0,n})<\i$ for all $n\geq N_0$.

From \eqref{Eqn4} and \eqref{Eqn6}, we obtain
\begin{eqnarray*}
  \|T^{m,n}-P_{n} \| &\leq & \|T^{k_0,n}T^{m,k_0-1}-T^{k_0,n}P_{k_0+1} \| + \| T^{k_0,n}P_{k_0+1}-P_{n}\| \\
   &=& \|T^{k_0,n}(T^{m,k_0-1}-P_{k_0+1}) \| + \| T^{k_0,n}P_{k_0+1}-P_{n}\|  \\
   &\leq & 2\d_P(T^{k_0,n}) + \i\\
   &\leq& 3\i , \ \ \  \ \textrm{for all} \ \ n\geq N_0
\end{eqnarray*}
hence, we deduce that $\|T^{m,n}-P\|\rightarrow 0$, as $n\rightarrow \infty$ which completes the proof.
\end{proof}

\begin{remark}
We note that if the generating sequence $\{T_{n}\}$ is stationary
(i.e. $T_{n}=T, \forall n\in \bn$), then the corresponding NDMC
reduces to a homogeneous chain. If for some projection $P$ with
$PT=TP=P$ and $P_{n}=P$, for all $n\in \bn$, then the conditions
of Theorem \ref{4.1} are satisfied, hence we obtain the
equivalence of the weak and uniform $P$-ergodicities for
homogeneous Markov chain $\{T^{n}\}$. This result has been proven
in \cite{MA} (namely by Corollary 4.7 and Proposition 4.10).
However, it is known that even in the case of $L^{1}$-spaces, the
weak $P$-ergodicity does not imply uniform $P$-ergodicity for NDMC
(see \cite{Paz}).
\end{remark}

\noindent Now, let us recall the following lemma which will be used to prove Theorem \ref{4.2}:
\begin{Lemma}\cite{M01}\label{F-Lemma-4.5}
Let $\{a_{j,n}\}$ be a sequence of real numbers such that
\begin{eqnarray*}\label{4a1}
&& 0\leq a_{j,n}\leq 1, \ \ \textrm{for all} \ \ j,n\in\bn,\\
\label{4a2} && a_{j,n}\leq a_{j,m}a_{m+1,n} \ \  \textrm{for all}
\ \ j\leq m\leq n.
\end{eqnarray*}
If there is a constant $K>0$ such that
\begin{equation*}\label{4a3}
\sum_{j=1}^n a_{j,n}\leq K \ \ \textrm{for all} \ \ n\in\bn,
\end{equation*}
then for each $j$ one has $a_{j,n}\to 0$ as $n\to\infty$.
\end{Lemma}
\begin{Theorem}\label{4.2} Let $(X,X_+,\ck,f)$ be an abstract state space and $\{T_n\}$ be
a generating sequence of the NDMC $\{T^{m,n}\}$. Let $\{P_n\}$ be
a left decreasing sequence of projections of $X$ such that
$T_nP_n=P_nT_n=P_n$, for all $n\in \mathbb{N}$, and
$P_n\rightarrow P$ in norm. If there exists a constant $C>0$ with
\begin{equation}\label{Eq.4}
  \sum_{l=1}^{n-1}\d_P(T^{l,n})\leq C,\ \forall n\in \mathbb{N},
\end{equation}
 then $\{T^{k,n}\}$ is
uniformly $P$-ergodic.
\end{Theorem}
\begin{proof}
For $j,n\in\bn$, put $a_{j,n}=\d_P(T^{j,n})$, and using Equation (\ref{Eqn5}), we have $PT_n=T_nP$. Then by Lemma
\ref{F-Lemma-4.5} and the hypothesis \eqref{Eq.4},  we find that
$\{T^{k,n}\}$ is weakly $P$-ergodic.

As $P_n\rightarrow P$, for any $\i>0$, there is $N_1\in\bn$ such that
\begin{equation}\label{422}
\|P_n-P\|<\i  \ \ \ \textrm{for all} \ \ n\geq N_1.
\end{equation}

One can see that
\begin{equation}\label{423}
\|T^{m,n}-P\|\leq \|T^{m,n}-P_{n}\|+\|P_{n}-P\|.
\end{equation}

\noindent Due to \eqref{422} it is enough to estimate
$\|T^{m,n}-P_{n+1}\|$. By \eqref{Eqn5}, one gets the following:
\begin{eqnarray*}\label{424}
\|T^{m,n}-P_{n+1}\| &=& \|T_{n}T^{m,n-1}-P_{n}\|\\
&\leq& \|T_{n}T^{m,n-1}-T_{n}P_{n-1}\|+\|T_{n}P_{n-1}-P_n\|\nonumber\\[2mm]
&= &\|T_{n+1}T_{n}T^{m,n-2}-T_{n}T_{n-1}P_{n-1}\|+\|T_{n}P_{n-1}-T_{n}P_{n}\|\nonumber\\[2mm]
&\leq &\|T^{n-1,n}T^{m,n-2}-T^{n-1,n}P_{n-1}\|+\d_P(T_n)\|P_{n-1}-P_n\|\nonumber\\[2mm]
&\leq&\|T^{n-1,n}T^{m,n-2}-T^{n-1,n}P_{n-2}\|+\|T^{n-1,n}P_{n-2}-T^{n-1,n}P_{n-1}\|\nonumber\\[2mm]
&&+\d_P(T_n)\|P_{n-1}-P_n\|\nonumber\\[2mm]
&\leq&\|T^{n-2,n}T^{m,n-3}-T^{n-2,n}P_{n-2}\|+\d_P(T^{n-1,n})\|P_{n-2}-P_{n-1}\|\nonumber\\[2mm]
&&+\d_P(T_n)\|P_{n-1}-P_n\|\nonumber\\[2mm]
&&\vdots\nonumber\\
&\leq&\|T_m-P_{m}\|\d_P(T^{m+1,n})+\sum_{\ell=m+1}^{n-1}\|P_{\ell-1}-P_\ell\|\d_P(T^{\ell,n}).
\end{eqnarray*}

Due to $\d_P(T^{m+1,n})\to 0$, there exists an $N_2\in\bn$ such
that $\d_P(T^{m+1,n})<\i$, for all $n\geq N_2$.

On other hand, from \eqref{422} we infer that
$\|P_{n-1}-P_n\|<2\i$ for all $n\geq N_1$. Without loss of
generality, we may assume that $N_1>m$. Hence,
\begin{eqnarray}\label{425}
\sum_{\ell=m+1}^{n-1}\|P_{\ell-1}-P_\ell\|\d_P(T^{\ell,n})&=&
\underbrace{\sum_{\ell=m+1}^{N_1}\|P_{\ell-1}-P_\ell\|\d_P(T^{\ell,n})}_{I_1}\nonumber
\\[2mm]
&&+
\underbrace{\sum_{j=N_1+1}^{n-1}\|P_{j-1}-P_j\|\d_P(T^{j,n})}_{I_{2}}
\end{eqnarray}

\noindent Let us estimate $I_1$ and $I_2$, separately.

We start with $I_2$. From \eqref{Eq.4} we easily find
\begin{eqnarray}\label{426}
I_2\leq 2\i\sum_{j=N_1+1}^{n-1}\d_P(T^{j,n})\leq 2\i C.
\end{eqnarray}

Now consider $I_1$. For any $\ell\in\{m+1,\ldots ,N_1\}$ one has
$\|P_{\ell-1}-P_\ell\|\leq 2$. The weakly $P$-ergodicity of
$T^{m,n}$ implies the existence of an $N_3\in\bn$ such that
$\d_P(T^{N_1,n})<\i$ for all $n\geq N_3$. Due to Equation
(\ref{Eqn5}) and by (vi) Theorem \ref{MA1-Thr. 3.7}, one can see
that
$$
\d_P(T^{\ell,n})\leq \d_P(T^{N_1,n})<\i.
$$
Therefore,
\begin{eqnarray}\label{427}
I_1=\sum_{\ell=m+1}^{N_1}\|P_{\ell-1}-P_\ell\|\d_P(T^{\ell,n})\leq
2\sum_{\ell=m+1}^{N_1}\d_P(T^{N_1,n})\leq 2(N_1-m)\i.
\end{eqnarray}
Hence, from \eqref{422}-\eqref{427} we obtain
$$
\|T^{m,n}-P\|\leq (3+C+2(N_1-m))\i
$$
for all $n\geq\max\{N_1,N_2,N_3\}$, which proves the assertion.
\end{proof}

\begin{Theorem}\label{4.3} Let $(X,X_+,\ck,f)$ be an abstract state space and $\{T_n\}$ be
a generating sequence of the NDMC. Let $\{P_n\}$ be a left decreasing sequence of projections of $X$ such that $T_nP_n=P_nT_n=P_n, \forall n\in \mathbb{N}$. Assume that
 $\{T^{k,n}\}$ is uniformly $P$-ergodic for a given projection $P$ and suppose there exist $k_n\in \bn$ and $\gamma_n\in [0,1)$ such that
 \[\d_{P_n}(T_n^{k_n})\leq \gamma_n \ \ \textrm{with} \ \ \sup_{n} \frac{k_n}{1-\gamma_n}<\infty.\]
  Then $P_n$ converges to $P$ in norm.
\end{Theorem}

\begin{proof}
As $\{T^{m,n}\}$ is uniformly $P$-ergodic, for each
$m\in\bn$, we have $\|T^{m,n}-P\|\to 0$ as $n\to\infty$.
Therefore, without loss of generality, we may assume $m=1$. Define
\begin{eqnarray}\label{43e}
E_n=T^{1,n}-T^{1,n-1},\ \ \ D_n=P_{n}-T^{1,n-1}.
\end{eqnarray}
One can see that
\begin{eqnarray*}
\|P_{n}-P\| &\leq&\|P_{n}-T^{1,n-1}\|+\|T^{1,n-1}-P\|\\
&=&\|D_n\|+\|T^{1,n-1}-P\|
\end{eqnarray*}
 so it is enough to show that
$\|D_n\|\to 0$ as $n\to\infty$, since $\|T^{1,n-1}-P\|\to 0$.

\noindent Notice that
\begin{equation*}
T_nD_n=T_nP_{n}-T_n T^{1,n-1}=P_{n}-T^{1,n}=D_n-E_n,
\end{equation*}
which yields
\begin{equation*}
T_n(T_nD_n+E_n)=T_n(D_n-E_n+E_n)=T_nD_n=D_n-E_n.
\end{equation*}
Then
\begin{equation}\label{432}
D_n=E_n+T_n^2D_n+T_nE_n
\end{equation}
Now iterating \eqref{432} $N$ times, one gets
\begin{eqnarray}\label{433}
D_n&=&E_n+T_nE_n+T_n^2E_n+T^3_nE_n+T_n^4D_n\nonumber\\[2mm]
&\vdots&\nonumber\\
&=&E_n+\sum_{j=1}^{2^N-1}T^j_nE_n+T^{2^N}_n D_n.
\end{eqnarray}

\noindent As $\{P_{n}\}$ is left decreasing sequence and for every $n\in
\bn,\ P_{n}T_{n}=P_{n}$,
\begin{eqnarray*}
P_{n}T^{1,n} &=& P_{n}T_{n}T_{n-1}\ldots T_{2}\\
&=& P_{n}T_{n-1}\ldots T_{2}\\
&=& P_{n}P_{n-1}T_{n-1}\ldots T_{2}\\
&=& P_{n}P_{n-1}T_{n-2}\ldots T_{2}\\
&=& P_{n}.
\end{eqnarray*}
Therefore,
\[
P_{n}E_{n}=P_{n}(T^{1,n}-T^{1,n-1})=P_{n}T^{1,n}-P_{n}T^{1,n-1}=P_{n}-P_{n}=0.
\]
Also, we have
\[ P_{n}D_{n}=P_{n}- P_{n}T^{1,n-1}=0.\]

\noindent Hence, by \eqref{433} and using (v) of Theorem
\ref{MA1-Thr. 3.7}, we obtain
\begin{eqnarray}\label{434}
\|D_n\| &\leq & \|E_n\|+\sum_{j=1}^{2^N-1}\|T^j_nE_n\|+\|T^{2^N}_n
D_n\|\nonumber\\[2mm]
&\leq&\|E_n\|+\sum_{j=1}^{2^N-1}\d_{P_{n}}(T^j_n)\|E_n\|+\d_{P_{n}}(T^{2^N}_n)\|D_n\|\nonumber\\[2mm]
&\leq&\|E_n\|\left (1+\sum_{j=1}^{2^N-1}\d_{P_{n}}(T^j_n)\right )+2\d_{P_{n}}(T^{2^N}_n).
\end{eqnarray}

By $\d_{P_n}(T_n^{k_n})\leq \gamma_n <1$, it follows that
$\d_{P_{n}}(T^{2^N}_n)<\i$ for a sufficiently large $N$. Moreover,
\begin{eqnarray*}\label{435}
\sum_{j=1}^{2^N-1}\d_{P_{n}}(T^j_n)&\leq&
\sum_{j=1}^{\infty}\d_{P_{n}}(T^j_n)\nonumber\\[2mm]
&=&\sum_{j=1}^{k_{n}}\d_{P_{n}}(T^j_n)+\sum_{j=k_{n}+1}^{2k_{n}}\d_{P_{n}}(T^j_n)+\cdots\nonumber\\[2mm]
&\leq & k_{n}+k_{n}\g_{n}+k_{n}\g_{n}^2+\cdots\nonumber\\[2mm]
&=&\frac{k_{n}}{1-\g_{n}}\leq K,
\end{eqnarray*}
Then, using \eqref{434}, we deduce
$$\|D_n\| \leq \|E_{n}\|(1+K)+2\i .$$
Now, according to $\|E_n\|\to 0$ as $n\to\infty $, one finds
$\|D_{n}\|\to 0$, which completes the proof.
\end{proof}

\begin{remark}
We stress that these types of results are even new in the case of
classical $L^1$-spaces. Moreover, if one considers abstract state
spaces associated with $C^*$-algebras, we get totaly new sort of
results which open new insight into the field of non-commutative
probability.
\end{remark}

\section{Weak $P$-ergodicity and the Doeblin condition}

In this section, we are going to investigate the weak
$P$-ergodicity of NDMC by means of an analogue of Doeblin's
condition on abstract state spaces.

Let $(X,X_+,\ck,f)$ be an abstract state space and let $P$ be a
projection on $X$. Now we are going to provide an analogue of
Doeblin's condition for NMMC associated with $P$ \cite{DP,M2013}, as follows:

{\bf Condition  ${\mathfrak{D}}_P$}.  We say that a NDMC
$\{T^{k,n}\}$ defied on $X$ satisfies \textit{condition
$\mathfrak{D}_P$}, if for every $k\in \bn$, there exist $\l_k\in
[0,1]$ and $n_k\in\N$, and for every $x,y\in \ck$, with $x-y\in
N_P$, one can find $z_k^{xy}\in \ck$,
 and $\f_{x,y}^k\in X_+$ with
$\sup\|\f_{x,y}^k\|\leq \frac{\l k}{2}$ such that
\begin{equation}\label{DC}
T^{k,k+n_k}x+\f_{x,y}^k\geq\l_k z_k^{xy}\ \ \text{and}\ \
T^{k,k+n_k}y+\f_{x,y}^k\geq\l_k z_k^{xy}.
\end{equation}

\vspace{5mm}

The next result is the main result of this section.

\begin{Theorem}\label{D}
Let $(X,X_+,\ck,f)$ be an abstract state space, $P$ be a Markov
projection on $X$, and $\{T_n\}$ be a generating sequence of the
NDMC $\{T^{k,n}\}$. Assume that $T_nP=PT_n$, for all $n\in
\mathbb{N}$. Then the following conditions are equivalent:
\begin{enumerate}
\item[(i)] the chain $\{T^{k,n}\}$ is weakly $P$-ergodic;
\item[(ii)] the chain $\{T^{k,n}\}$ satisfies the condition
$\mathfrak{D}_P$ with $\sum_{k=1}^\infty \l_k$ diverges;\\
\item[(iii)] for each $k\in\N$ there exist $\mu_k\in [0,1)$ and a
number $n_k\in\N$ such that
$$\d_P(T^{k,k+n_k})\leq \mu_k$$
with $\sum_{k=1}^\infty(1-\m_k)$ diverges.\\
\end{enumerate}
\end{Theorem}

\begin{proof} (i)$\Rightarrow$(ii): Assume that $\d_P(T^{k,n})\rightarrow 0$ as
$n\rightarrow \infty$. From the definition of $\d_P$, we have
\[\sup \{\|T^{k,n}x\|;\ \|x\|\neq 0\ \ \text{with}\ x\in N_P \}\rightarrow 0  \ \text{as}\ n\rightarrow \infty .\]
In particular, one finds
\[\sup \{\|T^{k,n}x-T^{k,n}y\|; \ |x-y|\neq 0\  \text{with}\ x-y\in N_P \}\rightarrow 0  \ \text{as}\ n\rightarrow \infty . \]

Now fix $y_0\in \ck$. Then there exists $n_k\in \bn$ such that for
all $x\neq y$ with $x-y_0\in N_P$,$y-y_0\in N_P$ we have
\[\|T^{k,k+n_k}x-T^{k,k+n_k}y_0\|< \frac{1}{4},\ \ \ \|T^{k,k+n_k}y-T^{k,k+n_k}y_0\|< \frac{1}{4}.\]
By means of the decomposition
\[T^{k,k+n_k}x-T^{k,k+n_k}y_0= (T^{k,k+n_k}x-T^{k,k+n_k}y_0)_{+}-(T^{k,k+n_k}x-T^{k,k+n_k}y_0)_{-}\]
\[T^{k,k+n_k}y-T^{k,k+n_k}y_0= (T^{k,k+n_k}y-T^{k,k+n_k}y_0)_{+}-(T^{k,k+n_k}y-T^{k,k+n_k}y_0)_{-}\]
let us define
\[\tilde{\f}_u^{(k)}:= (T^{k,k+n_k}u-T^{k,k+n_k}y_0)_{-},\ \text{where}\ u\in\{x,y\}.\]
Put
\[\f_{xy}^{(k)}:=\tilde{\f}_x^{(k)}+ \tilde{\f}_y^{(k)}. \]
Then, for all $x,y\in \ck$, with $x-y\in N_P$, we have
$\|\f_{xy}^{(k)} \|\leq \frac{1}{2}$. Therefore,
\begin{eqnarray*}
  T^{k,k+n_k}u + \f_{xy}^{(k)} &\geq & T^{k,k+n_k}u + \tilde{\f}_u^{(k)},\ u\in\{x,y\}\\
   &=& T^{k,k+n_k}y_0 + T^{k,k+n_k}u - T^{k,k+n_k}y_0 +  \tilde{\f}_u^{(k)}\\
   & \geq & T^{k,k+n_k}y_0.
\end{eqnarray*}
By defining $z_k^{xy}:=T^{k,k+n_k}y_0 $ and $\l_k=1$, the condition $\frak{D}_P$ is obtained.

  (ii)$\Rightarrow$(iii): For each $k\in \bn$, and for every $x,y\in \ck$ with $x-y\in N_P$, by the condition ${\mathfrak{D}}_P$, we have
  \[   T^{k,k+n_k}x+\f_{x,y}^k\geq\l_k z_k^{xy},\ \ T^{k,k+ n_k}y+\f_{x,y}^k\geq\l_k z_k^{xy},\]
with $\|\f_{x,y}^k\|\leq \frac{\l k}{2},\ \forall x,y\in \ck$. Then
\begin{eqnarray*}
 \|T^{k,k+n_k}x+\f_{x,y}^k-\l_k z_k^{xy}\|  &=& f(T^{k,k+n_k}x+\f_{x,y}^k-\l_k z_k^{xy}) \\
   &=& 1-(\l_k-f(\f_{x,y}^k)) \\
   &\leq&   1-\frac{\l_k}{2}.
\end{eqnarray*}
Similarly, $\|T^{k,k+n_k}y+\f_{x,y}^k-\l_k z_k^{xy}\|\leq 1-\frac{\l_k}{2}$. Let $c=\l_k-f(\f_{x,y}^k)$ and
\[ x_1:=\frac{1}{1-c}(T^{k,k+n_k}x+\f_{x,y}^k-\l_k z_k^{xy})\]
\[ y_1:=\frac{1}{1-c}(T^{k,k+n_k}y+\f_{x,y}^k-\l_k z_k^{xy}).\]
Then $x_1,y_1\in \ck$ and
\[ \|T^{k,k+n_k}x- T^{k,k+n_k}y \|=(1-c)\|x_1-y_1\|\leq 2(1-c),\] which implies that
\[\d_P(T^{k,k+n_k})\leq 1-\frac{\l_k}{2},\]
which proves (ii) by taking $\mu_k=1-\frac{\l_k}{2}$.

(iii)$\Rightarrow$(i): Given $k\in \bn$, then there exists
$\mu_k\in [0,1)$ and $n_k\in\N$ such that $$\d_P(T^{k,k+n_k})\leq
\mu_k.$$ Let $l_1=k+n_k$, by (ii),  one finds $n_{l_1}\in\bn$ and
$\mu_{l_1}\in [0,1)$ such that $\d_P(T^{l_1,l_1+n_{l_1}})\leq
\mu_{l_1}$. Continuing in the same argument, there exists a
sequence $\{l_j\}_{j=0}^\infty$, with $l_0=k$ and $\mu_{l_j}\in
[0,1)$ such that
 \[\d_P(T^{l_j,l_j+n_{l_j}})\leq \mu_{l_j}.\]
 For a large $n$, we define $L_n$ as
 \[L_n:= \max\{ j;\ l_j+n_{l_j}\leq n\}.\]
Due to the hypothesis of the theorem, we have $T^{k,n}P=PT^{k,n}$,
and then, from (vi) Theorem \ref{MA1-Thr. 3.7} one finds
 \begin{eqnarray*}
   \d_P(T^{k,n}) &=& \d_P(T^{l_{L_n},n}T^{l_{L-1},l_L}, \ldots T^{l_0,l_1}   ) \\
    &\leq & \prod_{j=1}^{L_n} \d_P(T^{l_{L_n-j},l_{L_n-j+1}}) \\
  &\leq & \prod_{j=1}^{L_n} \mu_{l_j} \rightarrow 0 \ \text{as}\ n\rightarrow
  \infty,
 \end{eqnarray*}
 since $\sum_{k=1}^\infty(1-\m_k)$ diverges,
hence (i) holds.
This completes the proof.
\end{proof}

\begin{remark}
We point out that the condition  $T_nP=PT_n$, $n\in\bn$ is only
used to establish the implication (iii)$\Rightarrow$(i). Hence, we
conclude that the implications
(i)$\Rightarrow$(ii)$\Rightarrow$(iii) are true without the stated
condition.
\end{remark}

\begin{remark} It is also worth to mention that an analogous kind of results to Theorem \ref{D} have been
established in \cite{HR,Paz,Rod} in the setting of $X=\ell_1$ and
$P$ is a one dimensional projection.
\end{remark}

As an application of the previous theorem, we have the following result.

\begin{Theorem}\label{TH.5.2}
  Let $(X,X_+,\ck,f)$ be an abstract state space, $P$ be a Markov projection on $X$, and for every $n\in\bn$ let $T_n\in \Sigma_P(X)$.
  Assume that $\|T_n-P\|< \epsilon_n$, where $\epsilon_n\rightarrow 0$. Then the NDMC $\{T^{m,n}\}$ is
  uniformly $P$-ergodic.
\end{Theorem}
\begin{proof}
  For every $n\in \bn$, let $Q_n=T_n-P$.  Then, $\|Q_n\|<\epsilon_n$ and $Q_nP=PQ_n$. For every $n\in \bn$,
  one finds
  \begin{eqnarray*}
    \d_P(T_n) &=& \sup_{x\in N_P, \|x\|=1}\|T_n(x)\| \\
     &=& \sup_{x\in N_P, \|x\|=1}\| Q_nx+Px \| \\
     &=& \sup_{x\in N_P, \|x\|=1}\| Q_nx\| \\
     &=& \d_P(Q_n)\\
     &\leq & \|Q_n\|<\epsilon_n.
  \end{eqnarray*}

Also, we have
  \begin{eqnarray*}
    PT_n &=& P(Q_n+P) \\
     &=& P^2Q_n+P \\
     &=& PQ_nP+P^3 \\
     &=& P(Q_n+P)P\\
     &=& PT_nP.
  \end{eqnarray*}
Now let us choose $n_k\in\bn$ such that $\epsilon_{k+n_k}\leq
\frac{1}{2^k}$. Then by Proposition \ref{TNp-in-Npmeans},
  \begin{eqnarray*}
    \d_P(T^{k,k+n_k}) &=& \d_P(T_{k+n_k}T_{k+n_k-1}\ldots T_{k+1}) \\
     &=& \d_P(T_{k+n_k})\d_P(T_{k+n_k-1})\ldots \d_P(T_{k+1}) \\
     &=& \d_P(T_{k+n_k})\\
     &\leq & \epsilon_{k+n_k}\\
     &\leq  & \frac{1}{2^k}.
     \end{eqnarray*}

Letting $\mu_k=\frac{1}{2^k}$, and using Theorem \ref{D}, we infer
that $\{T^{k,n} \}$ is weakly $P$-ergodic.
     Now, for every $n\in \bn$, put $P_n=P$ and then the three conditions (i), (ii) and (iii)
 of Theorem \ref{4.1} are satisfied which yields that $\{T^{m,n}\}$ is
  uniformly $P$-ergodic.
\end{proof}

Now, we are going to construct a left decreasing sequence of projections $\{P_n\}$ on $X$ which converges to a projection $P$ and to construct a sequence $\{T_n\}$ of uniformly $P_n$-ergodic Markov operators on $X$ such that the generated NDMC $\{T^{k,n}\}$ is uniformly $P$-ergodic.

\begin{example}\label{Ex.5.5}
Consider the space $\ell_1$, the subspaces $\ca =\{x\in \ell_1;\ x_{2n}=0\}$ and the operator $P:\ell_1 \to \ca$  defined by
\[P_1(x)= (x_1+x_2, 0, x_3+x_4, 0, x_5+x_6, 0 \ldots).\]
Then $P$ is a projection on $\ca$. Let $\cq_1:\ell_1 \to \ell_1$ be the operator defined by
\[\cq_1(x) = \left (\frac{-x_2}{2},\frac{x_2}{2}, \frac{-x_4}{2}, \frac{x_4}{2}, \frac{-x_6}{2}, \frac{x_6}{2}, \ldots \right ).\]
It is clear that $\cq_1^{n} \to 0$, as $n\rightarrow \infty$ so for some $n_0\in \bn$, we have $\|\cq_1^{n_0}\|<1$. Also, $P_1\cq_1=\cq_1 P_1=0$.
Then by Theorem 5.2 of \cite{MA}, we have that the operator $T_1=P_1+\cq_1$ is uniformly $P_1$-ergodic.
Similarly, define the operators $P_2$ and $\cq_2$ on $\ell_1$ by
\[P_2(x)= (0, 0, x_3+x_4, 0, x_5+x_6, 0, x_7+x_8,0 \ldots),\ \text{and}\ \cq_2(x)=\frac{1}{2}\cq_1(x).\]
One can see that $P_2\cq_2=\cq_2 P_2=0$ and $P_2\leq^l P_1$, with $T_2=P_2+\cq_2$ is uniformly $P_2$-ergodic.
Also, define the operators $P_3$ and $\cq_3$ on $\ell_1$ by
\[P_3(x)= (0, 0, 0, 0, x_5+x_6, 0, x_7+x_8,0 \ldots),\ \text{and}\ \cq_3(x)=\frac{1}{3}\cq_1(x),\]
and having $P_3\cq_3=\cq_3 P_3=0$ and $P_3\leq^l P_2$, with $T_3=P_3+\cq_3$ is uniformly $P_2$-ergodic.
Fixing some $N_0\in \bn$ and using the same argument, for every $n,\ 1\leq n\leq N_0$, put
\[P_n(x)= (0, 0,\ldots , 0,\overbrace{ x_{2n-1}+x_{2n}}^{(2n-1)^{th}-place},0,x_{2n+1}+x_{2n+2}  , 0 \ldots),\]
and for $n>N_0$, put $P_n=P_{N_0}$. Then $\{P_n\}_{n=1}^\infty$ is a left decreasing sequence of projections on $\ca$, which converges to $P=P_{N_0}$.
For every $n\in \bn$, define $\cq_n=\frac{1}{n}\cq_1$. It is clear that $\cq_n^m\rightarrow 0$, as $m\rightarrow \infty$ and $\|\cq_n\|<\frac{1}{n}$. Therefore, using Theorem 5.2 of \cite{MA}, the operator $T_n=P_n+\cq_n$ is uniformly $P_n$-ergodic, $\forall n\in \bn$ .

Considering the NHDC $\{T^{k,n}\}$ generated by $T_n$,
noting $T_n\rightarrow P$, as $n\rightarrow \infty$, we have $\|T_n-P\|<
\frac{1}{n}$, and hence by Theorem \ref{TH.5.2} we deduce that
$\{T^{k,n}\}$ is uniformly $P$-ergodic.

\end{example}

%
%
%
%

\section{Weakly $P$-ergodicity and its application to perturbations}

Let $(X,X_+,\ck,f)$ be an abstract state space, $T$ be a Markov
operator on $X$ and let $P$ be a projection on $X$. In \cite{MA},
we have proved that the homogenous Markov chain $\{T^{n}\}$ is
uniformly $P$-ergodic if and only if $TP=P$ and $T=P+Q$, where $Q$
is an operator on $X$ such that $PQ=QP=0$ and $\|Q^{n_0}\|<1$, for
some $n_0\in \bn$. Moreover,
 \[\d_P(T)\leq \|Q\|\leq 2\d_P(T).\]

Let $\{P_n\}$ be a sequence of projections of $X$. Then a NDMC $\{T^{m,n}\}$ is called {\it weakly ergodic w.r.t $\{P_n\}$}
if for every $m> 0$ one has
$$
\lim_{n\to\infty}\d_{P_{m+1}}(T^{m,n})=0.
$$

In this section, we are going to establish similar kind of results
for NDMC. Before, to formulate the main result of this section,
one needs some auxiliary facts.

\begin{Lemma}\label{lem-6.2}
 Let $(X,X_+,\ck,f)$ be an abstract state space and $\{T_n\}$ be
a generating sequence of NDMC $\{T^{k,n}\}$. Let $\{P_n\}$ be a
left decreasing sequence of projections of $X$ such that
$P_nT_n=P_n, \forall n\in \mathbb{N}$. If $\{T^{m,n}\}$ is weakly ergodic w.r.t $\{P_n\}$, then
\[\d_{P_{l}}(T^{m,n})\rightarrow 0,\ \text{as}\ n
\rightarrow \infty ,\ \text{for all}\ l\geq m+1.\]
\end{Lemma}

\begin{proof}
The condition $P_nT_n=P_n$ ($\forall n\in\bn$) together with the
left consistency of $\{P_n\}$, one finds $\ P_{k}T_{n}=P_{k}$, for
all $k\geq n$. Moreover,
\begin{eqnarray*}
P_{m+l+1}T^{m, m+l-1} &=& P_{m+l+1}(T_{m+l-1}T_{m+l-2}\ldots T_{m+1}) \\
&=& P_{m+l+1}T_{m+l-2}\ldots T_{m+1}\\
&=& P_{m+l+1}.
\end{eqnarray*}

\noindent Therefore, by Remark \ref{PT-PTP} and Proposition
\ref{TNp-in-Npmeans}, we have
\begin{eqnarray*}
\d_{P_{m+l+1}}(T^{m,n}) &=&\d_{P_{m+l+1}} (T^{m+l,n}T^{m,m+l-1}) \\
&\leq & \d_{P_{m+l+1}} (T^{m+l,n}) \d_{P_{m+l+1}} (T^{m,m+l-1})\\
 &\leq & \d_{P_{m+l+1}} (T^{m+l,n}) \rightarrow 0\ \text{as}\ n \rightarrow \infty
\end{eqnarray*}
hence the lemma is proved.
\end{proof}

\begin{Lemma}\label{lem-6.3}
Let $(X,X_+,\ck,f)$ be an abstract state space and $T$ be an
operator on $X$. If $\{P_n\}$ is a sequence of projections such
that $P_n\rightarrow P$ in norm, then $\d_{P_n}(T)\rightarrow
\d_P(T)$.
\end{Lemma}
\begin{proof}
 Given $\epsilon >0$, let $y=x-Px$, $\forall n\in \bn$, let $y_n=x-P_nx$ with $\|y_n\|\leq 1$. Then $y\in N_P$, $y_n\in N_{P_n}$ and $\|y\|\leq 1$. As $T$ is continuous, there exists $N_0\in \bn$ such that
  \[\| Ty_n-Ty\|< \epsilon ;\ \forall n>N_0.\]
  Therefore, $\forall n>N_0$ we have
  \[ \|Ty_n\|< \epsilon +\|Ty\|\leq \d_P(T) + \epsilon\]
  and then
  $\d_{P_n}(T)\leq \d_P(T) + \epsilon$, which implies that
  \[|\d_{P_n}(T)-\d_P(T)|\leq \epsilon ; \forall n \geq N_0\]
 this yields the required assertion.
\end{proof}

Now, we are ready to formulate the main result of this section.

\begin{Theorem}\label{Thr-6.1}
Let $(X,X_+,\ck,f)$ be an abstract state space, $\{P_n\}$ be a
left decreasing sequence of Markov projections of $X$ and
$\{T_n\}$ be a generation sequence of Markov operators of NDMC
$\{T^{m,n}\}$. Let $T_n\in \Sigma_{P_{n}}(X)$, for all
$n\in\bn$ with $T_{n}P_{n}=P_{n}$ and assume that
$T_{n}=P_{n}+Q_{n}$. Then the following statements are equivalent:
\begin{enumerate}
\item[(i)] for every $k\in \bn$, $\{T^{m,n}\}$ is weakly $P_k$-ergodic;
 \item[(ii)] $\{T^{m,n}\}$ is weakly ergodic w.r.t $\{P_n\}$;
 \item[(iii)] for every $m\in \bn$ one has
$$\| Q_{m+n}\ldots
Q_{m+1}\|\rightarrow 0, \ \ \textrm{as} \ \ n\rightarrow \infty.$$
Moreover,
\[ \d_{P_{m+1}}(T^{m,m+n})\leq \| Q_{m+n}\ldots  Q_{m+1}\| \leq 2\d_{P_{m+2}}(T^{m+1,m+n}).\]
\noindent Moreover, if $P_n\to P$ (in norm), then the above
statements are equivalent to \item[(iv)]  $\{T^{m,n}\}$ is weakly
$P$-ergodic.
\end{enumerate}
\end{Theorem}

\begin{proof}
(i)$\Rightarrow$(ii): Obvious as by the assumption $\{T^{m,n}\}$ is weakly $P_{m+1}$-ergodic.\\
(ii)$\Rightarrow$(i): Let us consider the following two cases: (a) For $0<k\leq m+1$, we have $P_{m+1}\leq^\ell P_k$ and by Proposition \ref{PQl}, we get
\[\d_{P_k}(T^{m,n})\leq \d_{P_{m+1}}(T^{m,n}),\]
which tends to zero by the assumption. (b) For $k > m+1$, and using Lemma \ref{lem-6.2}, we have that
\[\d_{P_{k}}(T^{m,n})\rightarrow 0,\ \text{as}\ n\rightarrow \infty, \]
hence we deduced (i).

To establish the implication (ii)$\Rightarrow$(iii), let us
first prove the following equality
\begin{equation}\label{Eq-1}
Q_{m+n}\ldots Q_{m+1}= T^{m, m+n}-T^{m+1,m+n}P_{m+1}.
\end{equation}
by induction on $n$.

For $n=2$, by noticing that
\[P_{m+2}T_{m+1}=P_{m+2}P_{m+1}T_{m+1}=P_{m+2}P_{m+1}=P_{m+2}.\]
we find
\begin{eqnarray*}
Q_{m+2}Q_{m+1} &=& (T_{m+2}-P_{m+2}) (T_{m+1}-P_{m+1})\\
&=& T_{m+2}T_{m+1}-T_{m+2}P_{m+1}-P_{m+2}T_{m+1}+P_{m+2}P_{m+1}\\
&=& T_{m+2}T_{m+1}-T_{m+2}P_{m+1}\\
&=& T^{m,m+2}-T^{m+1,m+2}P_{m+1}.
\end{eqnarray*}

Now assume the statement is true for $k=n$ and let us establish it
for $k=n+1$: First notice that
\begin{eqnarray*}
P_{m+n+1}T^{m,m+n} &=& P_{m+n+1}T_{m+n}\ldots T_{m+1}\\
&=& P_{m+n+1}P_{m+n}T_{m+n}\ldots T_{m+1}\\
&=& P_{m+n+1}P_{m+n}T_{m+n-1}\ldots T_{m+1}\\
&=& P_{m+n+1}T_{m+n-1}\ldots T_{m+1}\\
& \vdots & \\
&=& P_{m+n+1}.
\end{eqnarray*}
Similarly,
\[P_{m+n+1}T^{m+1,m+n}P_{m+1}=P_{m+n+1}.\]
Therefore,
\begin{eqnarray*}
Q_{m+n+1}\ldots Q_{m+1} &=& Q_{m+n+1}(Q_{m+n}\ldots Q_{m+1} )\\
&=& Q_{m+n+1}Q_{m+n}(T^{m, m+n}-T^{m+1,m+n}P_{m+1})\\
&=& T_{m+n+1}T^{m,m+n}-T_{m+n+1}T^{m+1,m+n}P_{m+1}\\
&-& P_{m+n+1}T^{m,m+n}+P_{m+n+1}T^{m+1,m+n}P_{m+1}\\
&=& T^{m,m+n+1}-T^{m+1,m+n+1}P_{m+1},
\end{eqnarray*}
which yields Equation \eqref{Eq-1}.

(ii)$\Rightarrow$(iii): By \eqref{Eq-1}, one finds
\begin{eqnarray*}
\|Q_{m+n}\ldots Q_{m+1}\| &=& \|T^{m, m+n}-T^{m+1,m+n}P_{m+1}\|\\
&=& \|T^{m+1, m+n}T_{m+1}-T^{m+1,m+n}P_{m+1}\|\\
&=& \|T^{m+1, m+n}(T_{m+1}-P_{m+1})\|\\
&\leq & \d_{P_{m+1}}(T^{m+1, m+n}) \|T_{m+1}-P_{m+1}\|\\
&\leq & 2\d_{P_{m+1}}(T^{m+1, m+n}) \ \text{(as both are Markov operators)}\\
& \leq & 2\d_{P_{m+2}}(T^{m+1, m+n}) \ \text{(using Proposition \ref {PQl})},
\end{eqnarray*}
which tends to $0$ by (ii).\\
(iii)$\Rightarrow$(ii): Again using \eqref{Eq-1}, we have
\[T^{m,m+n}= \prod _{ j=0 }^{n-1  }{Q_{m+n-j}  }+T^{m+1,m+n}P_{m+1},\]
then
\begin{eqnarray*}
\d_{P_{m+1}}(T^{m,m+n}) &=& \sup_{x\in N_{P_{m+1}}} \| \prod _{ j=0 }^{n-1  }{Q_{m+n-j} (x) }+T^{m+1,m+n}P_{m+1}(x)\|\\
&=& \sup_{x\in N_{P_{m+1}}} \| \prod _{ j=0 }^{n-1  }{Q_{m+n-j}  (x)}\|\\
&=& \d_{P_{m+1}}(\prod _{ j=0 }^{n-1  }{Q_{m+n-j}})\\
&\leq & \| \prod _{ j=0 }^{n-1  }{Q_{m+n-j}}\|\rightarrow 0, \
\text{as}\ n\rightarrow \infty \ \text{by (iii)}.
\end{eqnarray*}

(iv)$\Rightarrow$(ii): As $\{P_n\}$ is a left decreasing sequence and $P_{n}\rightarrow P$, by (ii) of Lemma \ref{PP}, one concludes that
$P$ is a projections and $P\leq^\ell P_{n}$ for all $n\in \bn$.
Therefore, due to Proposition \ref{PQl} one gets
\[\d_{P_{m+1}}(T^{m, n})\leq \d_{P}(T^{m, n})\rightarrow 0, \ \text{as}\ n \rightarrow \infty .\]

(ii)$\Rightarrow$(iv): Assume that $\d_{P_{m+1}}(T^{m,
n})\rightarrow 0$ as $n \rightarrow \infty$, then by Lemma
\ref{lem-6.2}, for all $ l\geq m+1$ we have that
\[\d_{P_{l}}(T^{m,n})\rightarrow 0,\ \text{as}\ n
\rightarrow \infty .\] Letting $l\rightarrow \infty$ and using
Lemma \ref{lem-6.3}, one gets
 \[\d_{P}(T^{m,n})\rightarrow 0,\ \text{as}\ n
\rightarrow \infty  \]
 which proves (iv), and hence the proof is completed.
\end{proof}

\begin{example}
  Consider the space $\ell_1$  and for every $n\in \bn$ let
  $$ e_n=( \underbrace{0,0,\ldots ,1}_{n}, 0, \ldots ).$$
  Construct the one dimensional projections $P_n:=T_{e_n}$. Then $\{P_n\}$ is a left decreasing sequence of projections (see Example \ref{LCE1}).
  Take any $r\in (0,1/2)$, and for
  every $n\in \bn$, define $T_n:=P_n+Q_n$, where $Q_n=r(I-P_n)$. Therefore,
  \[ \| Q_{m+n}\ldots
Q_{m+1}\| = r^n\| (I-P_{m+n})\ldots (I-P_{m+1})\|= (2r)^n\to 0, \ \text{as}\ n\to \infty .
  \]
Hence by Theorem \ref{Thr-6.1}, the generated chain $\{T^{m,n}\}$ is weakly ergodic w.r.t. $\{P_n\}$.
\end{example}

Next, we will discuss a perturbations of weakly and uniformly
$P$-ergodic chains. We need the following auxiliary fact.

\begin{Lemma}\label{Prop-6.6}
Let $(X,X_+,\ck,f)$ be an abstract state space and let
$\{T_n\},\{S_n\}$ be two generating sequences of NDMCs
$\{T^{k,n}\},\{S^{k,n}\}$ on $X$, respectively. Assume that
$$\sum_{n=1}^{\infty}\|T_n-S_n\|<\infty,$$ then for any
$\epsilon>0$, there exists $m_0\in \bn$ such that
  \[\|T^{m,n}-S^{m,n}\|< \epsilon; \ \forall m\geq m_0,\ \forall n>m.\]
\end{Lemma}

\begin{proof}
  For every $n\in \bn$, let $R_n=T_n-S_n$, and put $r_n=\|R_n\|$. Then
\begin{eqnarray*}
T^{m,n}&=&\prod_{j=0}^{n-m-1}T_{n-j}\\[2mm]
&=&\prod_{j=0}^{n-m-1}(S_{n-j}+R_{n-j})\\[2mm]
&=&\prod_{j=0}^{n-m-1}S_{n-j}+R_{m,n}=S^{m,n}+R_{m,n},
\end{eqnarray*}
where $R_{m,n}$ contains all possible products of $S_i$ and $R_i$
and keeping in mind $\|S_i\|=1$ for all $i$, one has
\begin{eqnarray*}\|R_{m,n}\|&\leq& \sum_i r_i + \sum_{i,j} r_ir_j +\sum_{i,j,k} r_ir_jr_k + \cdots  + \prod_{i=m+1}^{n}r_i\\[2mm]
&= &\prod_{i=m+1}^{n}(1+r_i)-1.
\end{eqnarray*}
  As $\sum_i r_i<\infty$, the product $\prod_{i=1}^{n}(1+r_i)$ converges. Hence, for any $\epsilon >0$, there exists $m_0\in \bn$
   such that $\|R_{m,n}\|<\epsilon, \forall m\geq m_0$ and $\forall n>m$, which completes the proof.
\end{proof}

Next result is about perturbations of weakly and uniformly
$P$-ergodicities of NDMC.

\begin{Theorem}\label{TSw}
Let $(X,X_+,\ck,f)$ be an abstract state space, $\{P_n\}$ be a
left decreasing sequence of Markov projections of $X$ which
converges to $P$ (in norm). Let $\{T_n\},\{S_n\}$ be two
generating sequences of NDMCs $\{T^{k,n}\},\{S^{k,n}\}$ on $X$,
respectively. Assume that $\forall n\in \bn$, $T_n, S_n\in
\Sigma_{P_{n}}(X)$ with $T_{n}P_{n}=P_{n}$, $S_{n}P_{n}=P_{n}$ and
$\sum_{n=1}^{\infty}\|T_n-S_n\|<\infty$. Then the following
statements hold:
\begin{enumerate}
\item[(i)] $\{T^{m, n}\}$ is uniformly $P$-ergodic if and only if
$\{S^{m, n}\}$ is uniformly $P$-ergodic; \item[(ii)] $\{T^{m,
n}\}$ is weakly $P$-ergodic if and only if $\{S^{m, n}\}$ is
weakly $P$-ergodic.
\end{enumerate}
\end{Theorem}

\begin{proof} (i).  Given any $\i >0$. As $\{T^{k,n}\}$ is uniformly
$P$-ergodic, there is $N_0\in\bn$ such that
$$
\|T^{k,n}-P\|<\i, \ \ \ \ n\geq N_0.
$$
On the other hand, by Proposition \ref{Prop-6.6}, there exists
$m_{0}\in \bn$ such that
 \begin{eqnarray}\label{TSw01}
\|T^{k,n}-S^{k,n}\|< \i, \ \ \forall k\geq m_0,\ \forall
n>k.\end{eqnarray} Therefore, if $k\geq m_0$, then one finds
 \begin{eqnarray}\label{TSw1}
 \|S^{k,n}-P\|\leq \|T^{k,n}-S^{k,n}\| + \|T^{k,n}-P\|<2\i\ \
 \ \forall n>N_0.
 \end{eqnarray}
 Now, if $0\leq k\leq m_0-1$, then by noticing $PS^{k,n}=P$ and \eqref{TSw01}, we
 find
 \begin{eqnarray*}
 \|S^{k,n}-P\|&=&\|S^{m_0,n}S^{k,m_0-1}-PS^{k,m_0-1}\|\\
 &=&
 \|(S^{m_0,n}-P)S^{k,m_0-1}\|\\
 &\leq &\|S^{m_0,n}-P\|<2\i
 \ \ \forall n>N_0.
 \end{eqnarray*}
This shows that  $\{S^{m, n}\}$  is uniformly $P$-ergodic. The
reverse can be proved by the same argument.

(ii). Assume that $\{T^{m, n}\}$ is weakly $P$-ergodic.
  For every $n\in \bn$, let us denote $Q_n=T_n-P_n$ and $\tilde{Q}_n=S_n-P_n$. By Theorem \ref{Thr-6.1}, the weak $P$-ergodicity of  $\{T^{m, n}\}$ implies
  \begin{equation}\label{Eq-13}
  \| Q_{m+n}\ldots  Q_{m+1}\|\rightarrow 0 , \ \text{as}\ n\rightarrow \infty ;\ \forall m\in \bn .
  \end{equation}
  Let us establish
     \begin{equation}\label{Eq-14}
  \| \tilde{Q}_{m+n}\ldots  \tilde{Q}_{m+1}\|\rightarrow 0 , \ \text{as}\ n\rightarrow \infty ;\ \forall m\in \bn .
  \end{equation}
  Using \eqref{Eq-1}, we have
\begin{eqnarray*}
&&Q_{m+n}\ldots Q_{m+1}= T^{m, m+n}-T^{m+1,m+n}P_{m+1},\\[2mm]
&&\tilde{Q}_{m+n}\ldots \tilde{Q}_{m+1}= S^{m,
m+n}-S^{m+1,m+n}P_{m+1} .\end{eqnarray*}

According to the weak $P$-ergodicity of $\{T^{k,n}\}$, there is
$N_1\in\bn$ such that $$\d_P(T^{m_0,m+n})<\i, \forall n\geq N_1.$$
Hence, by \eqref{TSw01} one finds
 \begin{eqnarray*}
 \|Q_{m+n}\ldots Q_{m+1}- \tilde{Q}_{m+n}\ldots \tilde{Q}_{m+1}\| &\leq & \|T^{m,m+n}-S^{m,m+n}\|
 \\[2mm]
 &&+ \|T^{m+1,m+n}-S^{m+1,m+n}\|\\[2mm]
  &\leq & \|T^{m_0,m+n}T^{m,m_0-1}-T^{m_0,m+n}S^{m,m_0-1}\|\\
 && + \|S^{m_0,m+n}S^{m,m_0-1}-T^{m_0,m+n}S^{m,m_0-1}\|
 \\[2mm]
 &&+ \|T^{m_0,m+n}T^{m+1,m_0-1}-T^{m_0,m+n}S^{m+1,m_0-1}\|\\
 &&+\|S^{m_0,m+n}S^{m+1,m_0-1}-T^{m_0,m+n}S^{m+1,m_0-1}\|\\[2mm]
  &\leq & \|T^{m_0,m+n}(T^{m,m_0-1}-S^{m,m_0-1})\|\\
 && + \|(S^{m_0,m+n}-T^{m_0,m+n})S^{m,m_0-1}\|
 \\[2mm]
 &&+ \|T^{m_0,m+n}(T^{m+1,m_0-1}-S^{m+1,m_0-1})\|\\
 &&+\|(S^{m_0,m+n}-T^{m_0,m+n})S^{m+1,m_0-1}\|\\[2mm]
  &\leq & 4\d_P(T^{m_0,m+n}) + 2\|S^{m_0,m+n}-T^{m_0,m+n}\|
 \\[2mm]
  & <&  6\i ,\ \ \ \textrm{for all} \ \  n>\max\{N_0,N_1\}.
 \end{eqnarray*}
 Consequently, by \eqref{Eq-13} we arrive at \eqref{Eq-14}, which yields that $\{S^{m, n}\}$
  is weakly $P$-ergodic.
\end{proof}

\begin{remark}
We stress that if one considers, as a particular case, the left
consistence sequence $\{P_n\}$ defined in Example \ref{LCE1}, then
the results of Theorem \ref{TSw} is even new in the
non-commutative setting. When $X=L^1$ and $P$ is a one dimensional
projection, in the mentioned setting, an analogue of Theorem
\ref{TSw} is known as the classical result \cite{Paz}.
\end{remark}

Let us provide some examples as application of the proved theorem.

\begin{example}\label{Ex.6.7}
Let $(X,X_+,\ck,f)$ be an abstract state space and $\{P_n\}$ be a
left decreasing sequence of Markov projections of $X$ which
converges to $P$ (in norm). Assume that $\{T^{k,n}\}$ be weakly
(resp. uniformly) $P$-ergodic NDMC which is generated by a
sequence of Markov operators $\{T_n\}_{n=1}^\infty$ such that
${T}_nP_n=P_n{T}_n=P_n, \forall n\in\bn$. Let $\{\l_n\} \subset
(0,1)$ such that
 \begin{equation}\label{1SnT}
\sum_{n=1}^\infty (1-\l_n)<\infty.
\end{equation}
Let us define a sequence $\{S_n\}$ of Markov operators by
$$
S_n=(1-\l_n)P_n+ \l_n{T}_n.
$$
It is clear that ${S}_nP_n=P_n{S}_n=P_n$, and we have
 \begin{eqnarray*}
   \|S_n-T_n\|=|1-\l_n|\|P_n-T_n\|\leq 2(1-\l_n)
    \end{eqnarray*}
which due to \eqref{1SnT} implies
 \begin{equation*}
\sum_{n=1}^\infty\|S_n-T_n\|<\infty.
 \end{equation*}
Hence, by Theorem \ref{TSw} we conclude that the NDMC
$\{S^{k,n}\}$ is weakly (resp.
uniformly) $P$-ergodic.
\end{example}

Now, let us provide a concrete example of uniformly $P$-ergodic
NDMC.

\begin{example}
Let $(X,X_+,\ck,f)$ be an abstract state space, and let $\{P_n\}$ be a
left decreasing sequence of Markov projections of $X$ which
converges to $P$ (in norm). Assume that
$\{\tilde{T}_n\}_{n=1}^\infty$ is a sequence of Markov operators
such that $\tilde{T}_nP_n=P_n\tilde{T}_n=P_n$. Let $\{\a_n\}_{n=1}^\infty
\subseteq (0,1)$ be such that
 \begin{equation}\label{2SnT}
\sum_{n=1}^\infty (1-\a_n) \ \ \textrm{diverges}.
\end{equation}
Let us define the sequence $\{T_n\}$ by
 \[ T_n=(1-\a_n)P+ \a_n\tilde{T}_n. \]
 Then, we have
 \begin{eqnarray*}
   T_nT_m &=& ((1-\a_n)P+ \a_n\tilde{T}_n)( (1-\a_m)P+ \a_m\tilde{T}_m) \\
    &=& (1-\a_n)(1-\a_m)P + \a_m(1-\a_n)P\tilde{T}_m + \a_n(1-\a_m)\tilde{T}_mP + \a_n\a_m\tilde{T}_n\tilde{T}_m \\
    &=& (1-\a_n\a_m)P+ \a_n\a_m\tilde{T}_n\tilde{T}_m.
 \end{eqnarray*}
 Hence $\forall k<n$
 \begin{equation*}
   T^{k,n} = T_nT_{n-1}\ldots T_{k+1}= \bigg(1-\prod_{j=k+1}^{n}\a_j\bigg)P+ \bigg(\prod_{j=k+1}^{n}\a_j\bigg)\tilde{T}_n\ldots \tilde{T}_{k+1}.
 \end{equation*}
    Therefore, due to \eqref{2SnT}, one finds
   \begin{equation*}
  \| T^{k,n}-P\| \leq \bigg(\prod_{j=k+1}^{n}\a_j\bigg) \|P-\tilde{T}_n\ldots \tilde{T}_{k+1}\|\leq 2\prod_{j=k+1}^{n}\a_j  \to
  0,
 \end{equation*}
as $n\to \infty$, which shows that $\{T^{k,n}\}$ is uniformly
$P$-ergodic.

Assume that a sequence $\{\b_n\}_{n=1}^\infty \subseteq (0,1)$ satisfies
\begin{equation}\label{3SnT}
\sum_{n=1}^\infty |\b_n-\a_n|<\infty.
\end{equation}
It is obvious that this together with \eqref{2SnT} yields that
 \begin{equation*}
\sum_{n=1}^\infty (1-\b_n) \ \ \textrm{diverges}.
\end{equation*}

Due to $P_n\to P$, we suppose (if it is needed by taking
subsequence) that
 \begin{equation}\label{4SnT}
\|P_n-P\|\leq \frac{1}{(1-\b_n)n^{1+\g}}, \ \ \textrm{for some} \
\ \g>0.
\end{equation}

 Now, let
us define the sequence $\{S_n\}$ as follows:
$$
S_n:=(1-\b_n)P_n+ \b_n\tilde{T}_n.
$$
It is clear that $\forall n\in\bn$, the operator $S_n$ is Markov
and $S_nP_n=P_nS_n=P_n$. Moreover,
\begin{eqnarray*}
S_n-T_n&=&(P_n-P)+(\a_nP-\b_nP_n)+(\b_n-\a_n)\tilde{T}_n\\[2mm]
&=&(P_n-P)+(\a_nP-\b_nP)+(\b_nP-\b_nP_n)+(\b_n-\a_n)\tilde{T}_n\\[2mm]
&=&(1-\b_n)(P_n-P)+(\a_n-\b_n)(P-T_n)
\end{eqnarray*}
and due to \eqref{3SnT},\eqref{4SnT}, we infer that
\begin{eqnarray*}
\sum_{n=1}^\infty\|S_n-T_n\|&\leq&
\sum_{n=1}^\infty(1-\b_n)\|P_n-P\|+\sum_{n=1}^\infty|\a_n-\b_n|\|P-T_n\|\\[2mm]
&\leq &
\sum_{n=1}^\infty\frac{1}{n^{1+\g}}+2\sum_{n=1}^\infty|\a_n-\b_n|<\infty .
\end{eqnarray*}
Hence, by Theorem \ref{TSw} we conclude that the NDMC
$\{S^{k,n}\}$ associated with $\{S_n\}$ is uniformly $P$-ergodic.
\end{example}

The above theorem can be generalized as follows:

\begin{Theorem}
Let $(X,X_+,\ck,f)$ be an abstract state space $\{P_n\}$ and
$\{\bar{P}_n\}$ be left decreasing sequences of Markov projections
of $X$ such that $\|P_{n}-\bar{P}_{n}\|\rightarrow 0$, and
$P_{n}\rightarrow P$ as $n\rightarrow \infty$. Suppose that the
generating sequences $\{T_n\}_{n=1}^\infty\subseteq \Sigma_{P_{n}}(X)$,
$\{S_n\}_{n=1}^\infty\subseteq \Sigma_{\bar{P}_{n}}(X)$ such that
$T_{n}P_{n}=P_{n}$, $S_{n}\bar{P}_{n}=\bar{P}_{n}$ and
$\sum_{n=1}^{\infty}\|T_n-S_n\|<\infty$. Then $\{T^{m, n}\}$ is
weakly $P$-ergodic if and only if $\{S^{m, n}\}$ is weakly
$P$-ergodic.
\end{Theorem}

\begin{proof}
As $\|P_{n}-\bar{P}_{n}\|\rightarrow 0$ and $P_{n}\rightarrow P$, we have $\bar{P}_{n}\rightarrow P$.
  For every $n\in \bn$, let us denote $Q_n=T_n-P_n$ and $\tilde{Q}_n=S_n-\bar{P}_n$.
  Using \eqref{Eq-1}, we have
 \begin{eqnarray*}
 &&Q_{m+n}\ldots Q_{m+1}= T^{m, m+n}-T^{m+1,m+n}P_{m+1},\\[2mm]
&& \tilde{Q}_{m+n}\ldots \tilde{Q}_{m+1}= S^{m,
m+n}-S^{m+1,m+n}\bar{P}_{m+1}.
\end{eqnarray*}
 Now, given any $\epsilon >0$, by Proposition \ref{Prop-6.6}, there exists $m_{0}\in \bn$ such that
 \[\|T^{m,n}-S^{m,n}\|< \epsilon; \ \forall m\geq m_0,\ \forall n>m.\]
 Also, there exists $m_{1}\in \bn$ such that
 \[\|P_{m}-\bar{P}_{m} \|< \epsilon , \ \forall m\geq m_{1}.\]
 Therefore,
 \begin{eqnarray*}
 \|Q_{m+n}\ldots Q_{m+1}- \tilde{Q}_{m+n}\ldots \tilde{Q}_{m+1}\| &\leq & \|T^{m,m+n}-S^{m,m+n}\| \\
 &+& |T^{m+1,m+n}P_{m+1}-S^{m+1,m+n}\bar{P}_{m+1}\| \\
 & \leq & \|T^{m,m+n}-S^{m,m+n}\|\\
 &+& \|T^{m+1,m+n}-S^{m+1,m+n}\| \\
 &+&\|P_{m+1}-\bar{P}_{m+1} \|.
 \end{eqnarray*}
Repeating the argument of the proof of Theorem \ref{TSw}, we
obtain the desired assertion.
\end{proof}

%
%
%
%
%

\section{Bair category results}

In this section, we are going to prove Bair category results to
describe the size of uniformly $P$-ergodic and weakly ergodic
w.r.t. $\{P_n\}$ NDMC's.

In the sequel, we will identify the generating sequence
$\ct=\{T_n\}_{n=1}^\infty$ of Markov operators with the
corresponding NDMC $\{T^{k,n}\}$.

Let $(X,X_+,\ck,f)$ be an abstract state space, and let $\{P_n\}$
be a left decreasing sequence of Markov projections of $X$. Let us
define
\begin{eqnarray*}
&&\frak{S}_{\{P_n\}}(X):= \big\{\ct=\{T_n\}_{n=1}^\infty: \ \
T_n\in
\Sigma_{P_n}(X)\ \text{and}\ T_nP_n=P_n, \forall n\in \bn \big\},\\
&&\frak{S}_{\{P_n\}}^w(X):= \big\{\ct\in \frak{S}_{\{P_n\}}(X): \
\ \{T^{m,n}\}\ \text{is weakly ergodic w.r.t.} \{P_n\} \big\}.
\end{eqnarray*}


 If $P$ is a any projection of $X$, then we define:
\begin{eqnarray*}
&&\frak{S}_{P}(X):=\big\{\ct=\{T_n\}_{n=1}^\infty: \ \ T_n\in
\Sigma_{P}(X)\ \text{and}\ T_nP=P, \forall n\in \bn \big\}, \\
&&\frak{S}_{P}^u(X):= \big\{\ct\in \frak{S}_{P}(X) \ \
\{T^{m,n}\}\ \text{is uniformly $P$-ergodic}\big\}.
\end{eqnarray*}

 The sets $\frak{S}_{\{P_n\}}(X)$ and $\frak{S}_{P}(X)$ are endowed with the following metric
\[d(\ct,\cs):= \sum_{n=1}^\infty \frac{\|T_n-S_n\|}{2^n}, \ \ \ct=\{T_n\}, \cs=\{S_n\}.\]

\noindent Then let us prove the following density theorems.

\begin{Theorem}
  Let $(X,X_+,\ck,f)$ be an abstract state space, and let $P$ be a Markov projection of $X$. Then $\frak{S}_{P}^u(X)$ is a dense subset of $\frak{S}_{P}(X)$.
\end{Theorem}

\begin{proof}
Let $\epsilon >0$ and let $\ct\in \frak{S}_{P}(X)$. For each $n\in
\bn$, put $S_{n}=\frac{\epsilon}{2}P+
(1-\frac{\epsilon}{2})T_{n}$. Clearly, $\{S_{n}\}$ is a sequence
of Markov operators on $X$ satisfying: $S_{n}P=PS_{n}=P$, and
$\cs=\{S_{n}\}\in \frak{S}_{P}(X)$. Moreover, one has
\[d(\cs,\ct)= \sum_{n=1}^\infty \frac{\|S_n-T_n\|}{2^n}= \frac{\epsilon}{2}\sum_{n=1}^\infty \frac{\|P-T_n\|}{2^n}< \epsilon .\]
Now, we claim that the generated NHMC $\{S^{k,n}\}$ is uniformly $P$-ergodic. Notice that
\[S^{k,n}=S_{n}\ldots S_{k+1}= \bigg(1-\big(1-\frac{\epsilon}{2}\big)^{n-k}\bigg)P+ \bigg(1-\frac{\epsilon}{2}\bigg)^{n-k}T_{n}\ldots T_{k+1}.\]
Hence, we have
\begin{eqnarray*}
\|S^{k,n}-P\| &\leq &  \bigg(1-\frac{\epsilon}{2}\bigg)^{n-k}\| P-T_{n}\ldots T_{k+1}\|\\
& \leq & 2  \bigg(1-\frac{\epsilon}{2}\bigg)^{n-k}\to 0,
\end{eqnarray*}
as $n\to\infty$. So, we arrive at $\cs\in \frak{S}_{P}^u(X)$ which
completes the proof.
\end{proof}

\begin{Theorem}\label{GGK}
  Let $(X,X_+,\ck,f)$ be an abstract state space, and let $\{P_n\}$ be a
left decreasing sequence of Markov projections of $X$. Then $\frak{S}_{\{P_n\}}^{w}(X)$ is a dense $G_\d$-subset of $\frak{S}_{\{P_n\}}(X)$.
\end{Theorem}

\begin{proof}
 Let $\epsilon >0$ and let $\ct=\{T_{n}\}\in \frak{S}_{\{P_n\}}(X)$. For each $n\in \bn$, we define
$$S_{n}=\frac{\epsilon}{2}P_{n}+ (1-\frac{\epsilon}{2})T_{n}.$$
As in Example \ref{Ex.6.7}, $\{S_{n}\}$ is a sequence of Markov
operators on $X$ satisfying:  $S_{n}P_{n}=P_{n}S_{n}=P_{n}$, and
$\cs=\{S_{n}\}\in \frak{S}_{\{P_n\}}(X)$. Moreover, one has
\[d(S,T)= \sum_{n=1}^\infty \frac{\|S_n-T_n\|}{2^n}= \frac{\epsilon}{2}\sum_{n=1}^\infty \frac{\|P_n-T_n\|}{2^n}< \epsilon .\]

For every $x$ and $y$ with $x-y\in N_{P_{n}}$, (for all $ n\in
\bn$) we have
\begin{eqnarray}\label{gg1}
  \|S_nx-S_ny\| &=& \|\frac{\epsilon}{2}P_{n}x+ (1-\frac{\epsilon}{2})T_{n}x-\frac{\epsilon}{2}P_{n}y- (1-\frac{\epsilon}{2})T_{n}y \| \nonumber \\
   &=& (1-\frac{\epsilon}{2})\|  T_n(x-y)\|\nonumber \\
   &\leq & (1-\frac{\epsilon}{2})\|(x-y)\|.
\end{eqnarray}

For all $m,n\in\bn$ with $m<$, one has $P_n(S^{m,n-1})=P_n$ which,
for all $x,y$ with $x-y\in N_{P_{m+1}}$, implies
$S^{m,n-1}x-S^{m,n-1}y \in N_{P_n}$. Indeed,
\[P_n(S^{m,n-1}x-S^{m,n-1}y)= P_n(x-y)=P_nP_{m+1}(x-y)=0.\]
Therefore, iterating \eqref{gg1}, we obtain
\begin{eqnarray*}
  \|S^{m,n}x-S^{m,n}y\| &=& \|S_n(S^{m,n-1}x-S^{m,n-1}y) \|\\
  &\leq & (1-\frac{\epsilon}{2}) \|  S^{m,n-1}x-S^{m,n-1}y\|\\
  & \vdots & \\
  & \leq & \big(1-\frac{\epsilon}{2}\big)^{n-m}\|(x-y)\|,
  \end{eqnarray*}
  which yields to
  \[\d_{P_{m+1}}(S^{m,n})\leq \big(1-\frac{\epsilon}{2}\big)^{n-m} \to 0,\ \text{as}\ n\to \infty ,\]
and this proves that $\{S^{m,n}\}$ is weakly ergodic w.r.t.
$\{P_n\}$. Hence, $\frak{S}_{\{P_n\}}^{w}(X)$ is a dense subset of
$\frak{S}_{\{P_n\}}(X)$.

Now for every $m\in \bn$, we have
\begin{eqnarray*}
 \d_{P_{m+1}}(T^{m,n+1})  &=&  \d_{P_{m+1}}(T_{n+1}T^{m,n})  \\
   &\leq & \d_{P_{m+1}}(T_{n+1}) \d_{P_{m+1}}(T^{m,n})\\
   &\leq & \d_{P_{m+1}}(T^{m,n}),
\end{eqnarray*}
which proves that $\{ \d_{P_{m+1}}(T^{m,n})\}_{n=1}^\infty$ is a
non-increasing sequence. Therefore,
\begin{equation}\label{Gd-SET}
\frak{S}_{\{P_n\}}^{w}(X)= \bigcap_{m=1}^\infty
\bigcap_{k=1}^\infty \bigcup_{n=m+1}^\infty  \left \{ \ct\in
\frak{S}_{\{P_n\}}(X);\ \d_{P_{m+1}}(T^{m,n})<\frac{1}{k} \right
\}.
\end{equation}
For all $m<n$, define the mapping
$\Phi_{m,n}:\frak{S}_{\{P_n\}}(X) \rightarrow [0,1]$ by
\[\Phi_{m,n}(\ct)=\d_{P_{m+1}}(T^{m,n}).\]
One can see that $\Phi_{m,n}$ is continuous, indeed
\begin{eqnarray*}
|\Phi_{m,n}(\ct)-\Phi_{m,n}(\cs)| &=&  |\d_{P_{m+1}}(T^{m,n})-\d_{P_{m+1}}(S^{m,n})|\\
   &\leq& \|T^{m,n}-S^{m,n} \|\\
   &=& \|T_nT_{n-1}\ldots T_{m+1}- S_nS_{n-1}\ldots S_{m+1}\| \\
   &\leq & \sum_{k=m+1}^n \| T_k-S_k\|\\
   &\leq & 2^n\sum_{k=1}^\infty \frac{ \| T_k-S_k\|}{2^k}\\
  &=& 2^nd(\ct,\cs).
\end{eqnarray*}

Hence, for every $k\in\bn$ the set
$\Phi_{m,n}^{-1}([0,\frac{1}{k}))$ is an open subset of
$\frak{S}_{\{P_n\}}(X)$, and from \eqref{Gd-SET}, we infer that
$\frak{S}_{\{P_n\}}^{w}(X)$ is a $G_\d$-subset of
$\frak{S}_{\{P_n\}}(X)$. This completes the proof.
\end{proof}

\begin{remark} We point out that the residual properties of
homogeneous Markov chains on $C^*$-algebras or von Nuemmann
algebras were intensively studied in \cite{BK,BK1,M0}. In the case
$X=L^1(\m)$, these results were investigated in \cite{B0,Iw}. An
analogue of Theorem \ref{GGK} has been proved in \cite{Iw,P} when
$X=\ell^1$. Therefore, our results are even new in the case of
$X=L^1(\m)$.
\end{remark}
%
%
%
%
\appendix

\section{Examples of Abstract State Spaces}

 Let us provide some examples of
abstract state spaces.

\begin{example}\label{E1}
  Let $M$ be a von Neumann algebra. Let $M_{h,*}$ be the
Hermitian part of the predual space $M_*$ of $M$. As a base $\ck$
we define the set of normal states of $M$. Then
$(M_{h,*},M_{*,+},\ck,\id)$ is a strong abstract state spaces,
where $M_{*,+}$ is the set of all positive functionals taken from
$M_*$, and $\id$ is the unit in $M$. In particular, if
$M=L^\infty(E,\m)$, then $M_{*}=L^1(E,\m)$ is an abstract state
space.
\end{example}

\begin{example}\label{E12}
 Let $A$ be a real ordered linear space and, as before,
let $ A_+$ denote the set of positive elements of $A$. An element
$e\in A_+$ is called \textit{order unit} if for every $a\in A$
there exists a number $\l\in\br_+$ such that $-\l e\leq a\leq\l
e$. If the order is Archimedean, then the mapping
$a\to\|a\|_e=\inf\{\l> 0\ : \ -\l e\leq a\leq\l e\}$  is a norm.
If $A$ is a Banach space with respect to this norm, the pair $(A,
e)$ is called \textit{an order-unit space with the order unit
$e$}. An element $\rho\in A^*$ is called \textit{positive} if
$\rho(x)\geq 0$ for all $a\in A_+$. By $A^*_+$ we denote the set
of all positive functionals. A positive linear functional is
called a \textit{state} if $\rho(e)=1$. The set of all states is
denoted by $S(A)$. Then it is well-known that $(A^*,A^*_+,S(A),e)$
is a strong abstract state space \cite{Alf}. In particular, if
$\ga_{sa}$ is the self-adjoint part of an unital $C^*$-algebra,
$\ga_{sa}$ becomes order-unit spaces, hence
$(\ga_{sa}^*,\ga_{sa,+}^*,S(\ga_{sa}),\id)$ is a strong abstract
state space.
\end{example}

\begin{example}\label{E13}
Let $X$ be a Banach space over $\br$. Consider a new Banach space
$\mathcal{X}=\br\oplus X$ with a norm
$\|(\a,x)\|=\max\{|\a|,\|x\|\}$. Define a cone
$\mathcal{X}_+=\{(\a,x)\ : \ \|x\|\leq \a, \ \a\in\br_+\}$ and a
positive functional $f(\a,x)=\a$. Then one can define a base
$\ck=\{(\a,x)\in\mathcal{X}:\ f(\a,x)=1\}$. Clearly, we have
$\ck=\{(1,x):\ \|x\|\leq 1\}$. Then
$(\mathcal{X},\mathcal{X}_+,\ck,f)$ is an abstract state space
\cite{Jam}. Moreover, $X$ can be isometrically embedded into
$\mathcal{X}$.  Using this construction one can study several
interesting examples of abstract state spaces.
\end{example}

\begin{example}\label{E14}  Let $A$ be the disc algebra, i.e. the sup-normed space
of complex-valued functions which are continuous on the closed
unit disc, and analytic on the open unit disc. Let $X =\{f\in A :\
f(1)\in\br \}$. Then $X$ is a real Banach space with the following
positive cone $X_+=\{f\in X: f(1)=\|f\|\}=\{f\in X: \
f(1)\geq\|f\|\}$. The space $X$ is an abstract state space, but
not strong one (see  \cite{Yo} for details).
\end{example}

\section{Examples of Markov operators}

 Let us consider several examples of Markov operators.

\begin{example}\label{E2}

Let $X=L^1(E,\m)$ be the classical $L^1$-space. Then any
transition probability $P(x,A)$ defines a Markov operator $T$ on
$X$, whose dual $T^*$ acts on $L^\infty(E,\m)$ as follows
$$
(T^*f)(x)=\int f(y)P(x,dy), \ \ f\in L^\infty.
$$
\end{example}

\begin{example}\label{E22}
Let $M$ be a von Neumann algebra, and consider
$(M_{h,*},M_{*,+},\ck,\id)$ as in Example \ref{E1}. Let $\Phi:
M\to M$ be a positive, unital ($\Phi(\id)=\id$) linear mapping.
Then the operator given by $(Tf)(x)=f(\Phi(x))$, where $f\in
M_{h,*}, x\in M$, is a Markov operator.
\end{example}

\begin{example}\label{E23} Let $X=C[0,1]$ be the space of real-valued continuous
functions on $[0,1]$. Denote
$$
X_+=\big\{x\in X: \ \max_{0\leq t\leq 1}|x(t)-x(1)|\leq 2
x(1)\big\}.
$$
Then $X_+$ is a generating cone for $X$, and $f(x)=x(1)$ is a
strictly positive linear functional. Then $\ck=\{x\in X_+: \
f(x)=1\}$ is a base corresponding to $f$. One can check that the
base norm $\|x\|$ is equivalent to the usual one
$\|x\|_{\infty}=\max\limits_{0\leq t\leq 1}|x(t)|$. Due to
closedness of $X_+$ we conclude that $(X,X_+,\ck,f)$ is an
abstract state space. Let us define a mapping $T$ on $X$ as
follows:
$$(Tx)(t)=tx(t).
$$
It is clear that $T$ is a Markov operator on $X$.
\end{example}

\begin{example}\label{E24} Let $X$ be a Banach space over $\br$. Consider the
abstract state space $(\mathcal{X},\mathcal{X}_+,\tilde \ck,f)$
constructed in Example \ref{E13}.  Let $T:X\to X$ be a linear
bounded operator with $\|T\|\leq 1$. Then the operator
$\mathcal{T}: \mathcal{X}\to\mathcal{X}$ defined by
$\mathcal{T}(\a,x)=(\a,Tx)$ is a Markov operator.
\end{example}

\begin{example}\label{E25} Let $A$ be the disc algebra, and let $X$ be the
abstract state space as in Example \ref{E14}. A mapping $T$ given
by $Tf(z)=zf(z)$ is clearly a Markov operator on $X$.
\end{example}

\bibliographystyle{amsplain}

\begin{thebibliography}{99}


\bibitem{Alf} E.M. Alfsen, \textit{Compact convex sets and booundary
integrals}, Springer-Verlag, Berlin, 1971.



%
\bibitem{B0}  W. Bartoszek, Norm residuality of ergodic operators, \textit{Bull. Pol. Ac. Sci. Math} {\bf  29}(1981), 165--167.

\bibitem{B} W. Bartoszek, Asymptotic properties of iterates
of stochastic operators on (AL) Banach lattices, \textit{Anal.
Polon. Math.} {\bf 52}(1990), 165-173.

\bibitem{BK} W. Bartoszek,  B. Kuna, On residualities in the set of Markov operators on $C_1$,
\textit{Proc. Amer. Math. Soc.} {\bf 133} (2005), 2119--2129.

\bibitem{BK1} W. Bartoszek, B. Kuna, Strong mixing Markov semigroups
on ${\mathcal C}_1$ are meager, \textit{Colloq. Math.} {\bf
105}(2006), 311--317.


\bibitem{BP} W. Bartoszek, M. Pulka, On mixing in the class of quadratic stochastic
operators, \textit{Nonlin. Anal.: Theor. Methods} {\bf 86} (2013),
95--113.

\bibitem{D} R. L. Dobrushin, Central limit theorem for
nonstationary Markov chains. I,II, \textit{Theor. Probab. Appl.}
{\bf 1}(1956),65--80; 329--383.

\bibitem{DP} C.C.Y. Dorea, A.G.C. Pereira, A note on a variation of
Doeblin's condition for uniform ergodicity of Markov chains,
\textit{Acta Math. Hungar.} {\bf 110}(2006), 287--292.

\bibitem{E}  E. Yu. Emelyanov, \textit{ Non-spectral asymptotic analysis of one-parameter operator semigroups},
Birkh\"{a}user Verlag, Basel, 2007.

\bibitem{EW2} E. Yu. Emel'yanov, M.P.H. Wolff, Asymptotic behavior of
Markov semigroups on non-commutative $L_1$-spaces, In book:
\textit{Quantum probability and infinite dimensional analysis}
(Burg, 2001), 77--83, QP--PQ: Quantum Probab. White Noise Anal., 15,
World Sci. Publishing, River Edge, NJ, 2003.

\bibitem{EW1} E. Yu. Emel'yanov, M.P.H. Wolff, Positive operators on
Banach spaces ordered by strongly normal cones,
\textit{Positivity} {\bf 7}(2003), 3--22.

\bibitem{EM2017}  N. Erkursun-Ozcan, F. Mukhamedov, Uniform ergodicities and perturbation bounds of Markov chains on
ordered Banach spaces, \textit{J. Phys.: Conf. Ser.} {\bf
819}(2017), 012015.

\bibitem{EM2018} N.  Erkursun-Ozcan, F. Mukhamedov,  Uniform ergodicities and
perturbation bounds of Markov chains on ordered Banach spaces,
\textit{Queast. Math.} {\bf 41} (2018), no. 6,  863--876.

\bibitem{EM20181}  N.  Erkursun-Ozcan, F. Mukhamedov, Uniform ergodicities of Lotz -
R\"abiger nets of Markov operators on ordered Banach spaces,
\textit{Results Math.} {\bf 73} (2018), no. 1, 35.

\bibitem{GQ} S. Gaubert, Z. Qu, Dobrushin's ergodicity
coefficient for Markov operators on cones and beyond,
\textit{Integ. Eqs. Operator Theor.} {\bf 81}(2014), 127--150.



\bibitem{HB} J. Hajnal, Weak ergodicity in nonhomogeneous Markov
chains, \textit{Proc. Cambridge Phil. Soc.} {\bf 54}(1958)
233--246.

\bibitem{H}  D.J. Hartfiel. Coefficients of ergodicity for imprimitive marices, \textit{Commun. Statis. Stochastic Models} {\bf 15}(1999), 81--88.

\bibitem{HR}  D.J. Hartfiel, U.G. Rothblum,  Convergence of inhomogeneous products of matrices and coefficients of ergodicity,
\textit{Lin. Alg. Appl.} {\bf 277}(1998), 1--9.

\bibitem{HH} H. Hennion, L. Harve, \textit{Limit theorems for Markov chains and stochastic properties of dynamical systems by
quasi-compactness}, Lec. Notes Math.  {\bf 1766} (2001), Springer-Verlag, Berlin.

\bibitem{IS} I.C.F. Ipsen, T.M. Salee, Ergodicity coefficients
defined by vector norms, \textit{SIAM J. Matrix Anal. Appl.} {\bf
32}(2011), 153--200.

\bibitem{Iw} A. Iwanik, Baire category of mixing for stochastic operators, \textit{Rend. Circ. Mat. Palermo, Serie II}
{\bf 28} (1992), 201--217.

\bibitem{Jach} J. Jachymski, Convergence of iterates of linear operators and the Kelisky-
Rivlin type theorems, \textit{Studia Math.}, {\bf 195} (2009), 99--113.


\bibitem{J1} R.Jajte, \textit{Strong linit theorems in non-commutative probability}, Lecture
Notes in Math. vol. 1110, Berlin-Heidelberg: Springer 1984.


\bibitem{Jam} G. Jameson, \textit{Ordered linear spaces}, Lect.
Notes  Math. V. 141, Springer-Verlag, Berlin, 1970.

\bibitem{JI} J.Johnson, D. Isaacson, Conditions for strong ergodicity
using intensity matrices, \textit{J. Appl. Probab.} {\bf 25}(1988)
34--42.

\bibitem{KM} I. Kontoyiannis, S.P. Meyn, Geometric ergodicity and the spectral gap of non-reversible Markov chains,  \textit{Probab. Theory Relat. Fields}
{\bf 154}(2012), 327--339.

\bibitem{K} U. Krengel, \textit{Ergodic Theorems}, Walter de Gruyter,
Berlin-New York, 1985.

%


\bibitem{MI} R. W. Madsen, D. L. Isaacson, Strongly ergodic behavior for non-stationary
Markov processes, \textit{Ann. Probab.} {\bf 1} (1973), 329--335.


\bibitem{MT}  S.P. Meyn, R.L. Tweedie, \textit{Markov chains and stochastic stability}, Springer-Verlag, 1996.

%

\bibitem{M} F. Mukhamedov, Dobrushin ergodicity coefficient and ergodicity of noncommutative
Markov chains, \textit{J. Math. Anal. Appl.} {\bf 408} (2013),
364--373.

\bibitem{M2013} F. Mukhamedov, On $L_1$-Weak Ergodicity of nonhomogeneous discrete Markov processes
and its applications, \textit{Rev. Mat. Compult.} {\bf 26}(2013),
799--813.

\bibitem{M0} F. Mukhamedov, Ergodic properties of nonhomogeneous Markov chains defined on ordered Banach spaces with a
base, \textit{Acta. Math. Hungar.} {\bf 147} (2015), 294--323.

\bibitem{M01} F. Mukhamedov, Uniform stability and weak ergodicity of nonhomogeneous Markov
chains defined on ordered Banach spaces with a base,  \textit{Positivity}, {\bf 20}(2016), 135--153.

\bibitem{MA} F. Mukhamedov A. Al-Rawashdeh, On generalized Dobrushin ergodicity coefficient and uniform Ergodicites of Markov Operators,
\textit{Positivity},  DOI 10.1007/s11117-019-00713-0 (2019).



\bibitem{Paz} A. Paz, Ergodic theorems for infinite probabilistic
tables, \textit{Ann. Math. Statist.} {\bf 41}(1970), 539--550.

\bibitem{P} M. Pulka, On the mixing property and the ergodic principle for
nonhomogeneous Markov chains, \textit{Linear Alg. Appl.} {\bf 434}
(2011), 1475--1488.


\bibitem{Rod} A. Rhodius, On ergodicity coefficients of infinite stochastic
matrices, \textit{Zeit. Anal.  Anwen.} {\bf 19}(2000), 873--887.

\bibitem{SZ} T.A. Sarymsakov, N.P.
Zimakov, Ergodic principle for Markov semi-groups in ordered
normal spaces with basis, \textit{Dokl. Akad. Nauk. SSSR} {\bf
289} (1986), 554--558.


\bibitem{Se} E. Seneta, On the historical
development of the theory of finite inhomogeneous Markov chains,
\textit{Proc. Cambridge Philos. Soc.}, {\bf 74} (1973), 507--513

\bibitem{Se2} E. Seneta, \textit{Non-negative matrices and Markov chains},
Springer, Berlin, 2006.

\bibitem{SW} O. Szehr, M.M. Wolf, Perturbation bounds for quantum
Markov processes and their fixed points, \textit{J. Math. Phys.}
{\bf 54}(2013), 032203.

\bibitem{T} Ch. P. Tan, On the weak ergodicity of nonhomogeneous Markov chains,
{\it Statis. \& Probab. Lett.} {\bf 26}(1996), 293--295.


\bibitem{Yo} D. Yost, A base norm space whose cone is not 1-
generating, \textit{Glasgow Math. J.} {\bf 25} (1984), 35--36.

\bibitem{ZI} A. I. Zeifman,  D. L. Isaacson,
 On strong ergodicity for nonhomogeneous continuous-time Markov chains,
 \textit{Stochast. Process. Appl.} {\bf 50}(1994), 263--273.

\bibitem{VW} J. J. Vardy, B. A. Watson,
Markov processes on Riesz spaces, \textit{Positivity} {\bf 16}
(2012), 373--391.

\bibitem{WN}  Y. C. Wong, K. F. Ng, \textit{Partially ordered topological vector spaces}, Clarendon Press, 1973.


\end{thebibliography}

\end{document}